\documentclass[final,leqno,onefignum,onetabnum]{siamltex1213}

\usepackage{rotating}
\usepackage{url}
\usepackage{float}
\usepackage{graphicx}
\usepackage{color}
\usepackage{latexsym}
\usepackage{amsfonts}
\usepackage{amsmath}
\usepackage{psfrag}
\usepackage{mathrsfs}
\usepackage{algpseudocode}
\usepackage[ruled]{algorithm}

\usepackage{caption}
\usepackage{subcaption}

\newtheorem{remark}{Remark}

\def\R{\mathbb{R}}

\def\N{\mathbb{N}}

\def\meno{\par\medskip\noindent}

\def\eps{\epsilon}

\newtheorem{ass}{Assumption}

\title{A Fast  Active Set Block Coordinate Descent Algorithm for $\ell_1$-regularized least squares}
\author{ M. De Santis$^\dag$, S. Lucidi$^\ddag$, F. Rinaldi$^*$}

\author{Marianna~De~Santis\thanks{Fakult\"at f\"ur Mathematik, Technische Universit\"at Dortmund, Vogelpothsweg 87, 44227 Dortmund, Germany
\email{marianna.de.santis@tu-dortmund.de}}
  \and
  Stefano~Lucidi\thanks{Dipartimento di Ingegneria Informatica
    Automatica e Gestionale, Sapienza Universit\`{a} di Roma, Via
    Ariosto, 25, 00185 Roma, Italy
    \email{stefano.lucidi@dis.uniroma1.it}}
  \and
  Francesco~Rinaldi\thanks{Dipartimento di Matematica, Universit\`a
    di Padova, Via Trieste, 63, 35121 Padova, Italy
    \email{rinaldi@math.unipd.it}}
}

\begin{document}
\maketitle
\slugger{mms}{xxxx}{xx}{x}{x--x}

\begin{abstract}
The problem of finding sparse solutions to underdetermined systems of linear equations arises
in several applications (e.g. signal
and image processing, compressive sensing, statistical inference).
A standard tool for dealing with sparse recovery is the
$\ell_1$-regularized least-squares approach that has
been recently attracting the attention of many researchers. 

In this paper, we describe an active set estimate (i.e. an estimate of the indices of the zero variables in the optimal solution) for the
considered problem that tries to  quickly identify as many active variables as possible at a given point, while guaranteeing that some approximate 
optimality conditions are satisfied. A relevant feature of the estimate is that it gives a significant reduction of the objective function 
when setting to zero all those variables estimated active. This enables to easily embed it into a given globally converging 
algorithmic framework.

In particular, we include our estimate into a block coordinate descent algorithm for $\ell_1$-regularized least squares,  analyze the convergence properties of 
this new active set method, and prove that its basic version converges with linear rate. 

Finally, we report some numerical results showing the effectiveness of the approach.


\end{abstract}



    \pagestyle{myheadings}
    \thispagestyle{plain}
    \markboth{M. DE SANTIS, S. LUCIDI, F. RINALDI}{AN ACTIVE SET BCD METHOD FOR $\ell_1$-REGULARIZED LEAST SQUARES}

\section{Introduction}\label{introduction}
The problem of finding sparse solutions to large underdetermined linear systems of equations has received a
lot of attention in the last decades. This is due to the fact that several real-world applications 
can be formulated as linear inverse problems.
A standard approach is the so called $\ell_2$-$\ell_1$ unconstrained optimization problem:
\begin{equation}\label{l2l1}
\min_{x\in \R^n} \, \frac{1}{2}\|Ax - b\|^2 + \tau \|x\|_1,
\end{equation}
where $A\in \R^{m\times n}$, $b\in \R^m$, $x\in \R^n$ $(m<n)$ and $\tau\in \R^+$. We denote by $\|\cdot\|$ the standard
$\ell_2$ norm and by  $\|\cdot\|_1$ the $\ell_1$ norm defined as $\|x\|_1=\sum_{i=1}^n |x_i|$.

Several classes of algorithms
have been proposed for the solution of Problem~(\ref{l2l1}). Among the others, we would like to remind Iterative
Shrinkage/Thresholding (IST) methods (see e.g. \cite{FISTA,TwIST,CombettesWajs,Daubechies,SPARSA}), Augmented Lagrangian
Approaches (see e.g. \cite{salsa}), Second Order Methods (see e.g. \cite{nocedal,gondzio}), Sequential Deterministic  (see e.g.
\cite{tseng0,tseng,yun}) and Stochastic (see e.g. \cite{fountoulakis,ric2} and references therein) Block Coordinate Approaches, Parallel
Deterministic (see e.g. \cite{facchinei} and references therein) and Stochastic (see e.g. \cite{facchineiNonconvex,ric} and references therein)
Block Coordinate Approaches, and Active-set strategies  (see e.g. \cite{GriesseLorenz,Wenetal1,Wenetal2}).

The main feature of this class of problems is the fact that the optimal solution is usually very sparse (i.e. it has many zero
components). Then, quickly building and/or correctly identifying the active set (i.e. the subset of zero components in an
optimal solution) for Problem~(\ref{l2l1}) is becoming a crucial task in the context of Big Data Optimization, since  it can
guarantee relevant savings in terms of CPU time. As a very straightforward example, we can consider a huge scale problem having a
solution with just a few nonzero components. In this case, both the fast construction and the correct identification of the
active set can considerably reduce the complexity of the problem, thus also giving us the chance to use  more sophisticated
optimization
methods than the ones usually adopted.Various attempts have been made in order to use active set technique in the context of $\ell_1$-regularized problems. 

In \cite{Wenetal1,Wenetal2}, Wen et al. proposed a two-stage algorithm, FPC-AS, where an estimate of the active variables set is
driven by using a first-order iterative shrinkage method. 

In \cite{Wright}, a block-coordinate relaxation approach with proximal linearized subproblems yields convergence to critical
points, while identification of the optimal manifold (under a nondegeneracy condition) allows acceleration techniques
to be applied on a reduced space.

In \cite{Hsieh}, the authors solve an $\ell_1$-regularized log determinant program related to the problem of sparse inverse
covariance matrix estimation combining a second-order approach with a technique to correctly identifying the active set.

An efficient version of the  two-block nonlinear constrained Gauss-Seidel algorithm that at each iteration
fixes some variables to zero according to a simple active set  rule has been proposed
in \cite{porcelli} for solving $\ell_1$-regularized least squares.

In a recent paper \cite{nocedal}, Nocedal et al. described an interesting family of second order methods for
$\ell_1$-regularized convex problems. Those methods  combine  a semi-smooth Newton approach
with a mechanism to identify the active manifold in the given problem. 

In the case one wants to solve very large problems, Block Coordinate Descent Algorithms (both Sequential and Parallel)
represent a very good alternative and, sometimes, the best possible answer \cite{tseng}. An interesting Coordinate Descent algorithm combining
a Newton steps with a line search technique was described by  Yuan et al. in \cite{yuan}. In this context, the authors also
proposed a shrinking technique (i.e. a heuristic strategy that tries to fix to zero a subset of variables according to a certain
rule), which can be seen as a way to identify the active variables. In \cite{tseng}, some ideas on how to  speed up their Block
Coordinate Descent Algorithm by including an active set  identification strategy are described, but no theoretical analysis is
given for the resulting approach.

What we want to highlight here is that all the approaches listed above, but the one described in \cite{nocedal}, 
estimate the final active set by using the current active set and perform subspace minimization on the remaining variables. 
In \cite{nocedal}, the authors define an estimate
that performs multiple changes in the active manifold by also including variables that are nonzero at a given point and satisfy
some specific condition. Since this active set mechanism, due to the aggressive changes in the index set, can cause 
cycling, including the estimate into a globally converging algorithmic framework is not always straightforward. 

In this work, we adapt the active set estimate proposed in \cite{facchineilucidi} for constrained optimization problems to the
$\ell_1$-regularized least squares case. Our estimate, similarly to the one proposed in \cite{nocedal}, does not only focus on
the zero variables of a given point. Instead it tries to quickly identify as many active variables as possible (including  the
nonzero variables of the point), while guaranteeing that some approximate optimality conditions are satisfied. 

The main feature of the proposed active set  strategy is that
 a significant reduction of the objective function is obtained when setting to zero all those variables estimated active.
This global property, which is strongly related to the fact that the components estimated active satisfy an approximate optimality condition, makes 
easy to use the estimate into a given globally converging algorithmic framework.

Furthermore, inspired by the papers \cite{tseng,yuan,yun}, we describe a new Block Coordinate Descent Algorithm that embeds 
the considered active set  estimate. At each iteration, the method first sets to zero the active variables, then uses a 
decomposition strategy for updating a bunch of the non-active ones. 
On the one hand, decomposing the non-active variables enables to handle huge scale problems
that other active set approaches cannot solve in reasonable time.
On the other hand, since the subproblems analyzed at every iteration explicitly take into account the 
$\ell_1$-norm, the proposed algorithmic framework  does not require a sign identification strategy
(for the non-active variables), which is tipically needed when using other active set methods from the literature.

The paper is organized as follows. In Section~\ref{estimate}, we introduce our active set strategy. In Section~\ref{algorithm},
we describe the active set coordinate descent algorithm, and prove its convergence. We further analyze the convergence rate of the algorithm.
In Section~\ref{numres}, we report some
numerical results showing the effectiveness of the approach. Finally, we draw some conclusions in Section~\ref{conclusions}.
\section{Notation and Preliminary Results}\label{prelim}
Throughout the paper we denote  by $f(x)$, $q(x)$, $g(x)$ and $H$ the original function in Problem~\eqref{l2l1}, the quadratic term of 
the objective function in Problem~\eqref{l2l1}, the $n$ gradient vector
and the $n\times n$ Hessian matrix of $\displaystyle\frac{1}{2} \|Ax - b\|^2$ respectively.
Explicitly
$$q(x)=\frac{1}{2} \|Ax - b\|^2, \quad g(x)= A^\top(Ax - b), \quad H = A^\top A.$$
Given a matrix $Q \in \R^{n \times n}$, we further denote by $\lambda_{max}(Q)$ and $\lambda_{min}(Q)$ the maximum 
and the minimum eigenvalue of the matrix $Q$, respectively. Furthermore, with $I$ we indicate the set of indices 
$I=\{1,\dots,n\}$, and with 
$Q_{I_j I_j}$ we indicate the submatrix of $Q$ whose rows and columns indices are in $I_j\subseteq I$. 
We also report the optimality conditions for Problem~\eqref{l2l1}:
\begin{proposition}
 $x^\star\in \R^n$ is an optimal solution of Problem~\eqref{l2l1} if and only if
 \begin{equation}\label{optcondorpr}
 \left \{
 \begin{array}{ll}
  x_i^\star>0, &g_i(x^\star)+\tau=0\\
  x_i^\star<0, &g_i(x^\star)-\tau=0\\
  x_i^\star=0, & -\tau\leq g_i(x^\star)\leq\tau.
 \end{array}
\right .
 \end{equation}
 \end{proposition}
 Furthermore, we define a continuous function $\Phi_i(x)$ that measures the violation of the optimality conditions in $x_i$ (and is connected
to the Gauss-Southwell-r rule proposed in \cite{tseng}), that is
\begin{equation}\label{tsengest}
\Phi_i(x)=-\mbox{mid}\left\{ \frac{g_i(x)-\tau}{H_{ii}}, x_i, \frac{g_i(x)+\tau}{H_{ii}}\right\},
\end{equation}
where mid$\{a,b,c\}$ indicates the median of $a,\ b,\ c$.

 Finally, we recall the concept of strict complementarity.
 \begin{definition}\label{stc}
 Strict complementarity holds if, for any $x^\star_i=0$, we have
 \begin{equation}
  -\tau<g_i(x^\star)<\tau.
 \end{equation}

 \end{definition}
\section{Active set estimate}\label{estimate}
All the algorithms that adopt active set strategies need to estimate a particular subset of components of the optimal solution $x^\star$.
In nonlinear constrained minimization problems, for example, using an active set strategy usually means correctly identifying the set of active constraints at the solution.
In our context, we deal with Problem~(\ref{l2l1}) and the active set is considered as the subset of zero-components of $x^\star$.
\begin{definition}\label{def:activeset}
Let $x^\star\in \R^n$ be an optimal solution for Problem~(\ref{l2l1}).
We define the active set as follows:
\begin{equation}\label{AS}
{\cal \bar A}(x^\star)=\big\{i\in I: x^\star_i=0\big\}.
\end{equation}
We further define as non-active set the complementary set of ${\cal \bar A}(x^\star)$:
\begin{equation}\label{NAS}
 {\cal \bar N}(x^\star)=I\setminus {\cal \bar A}(x^\star)=\big\{i\in \{1,\ldots, n\}:x^\star_i\neq 0\big\}.
\end{equation}
\end{definition}
In order to get an estimate of the active set we rewrite Problem~(\ref{l2l1}) as a box constrained programming problem and we use
similar ideas to those proposed in \cite{DeSantisDiPilloLucidi}.

Problem~(\ref{l2l1}) can be equivalently rewritten as follows:
\begin{equation}
\begin{array}{l l}\label{prob1}
\min & \frac{1}{2}\|A(u-v)-b\|^2 +\tau \sum_{i=1}^n (u_i+v_i)\\
& u\ge 0\\
& v\ge 0,
\end{array}
\end{equation}
where $u,v \in \R^n$.
Indeed, we can transform a solution $x^\star\in \R^n$ of Problem~\eqref{l2l1} into a solution $(u^\star, v^\star)\in \R^n\times \R^n$ of \eqref{prob1}  by using the following transformation:
$$
\begin{array}{l}
u^\star=\max(0,x^\star),\qquad v^\star=\max(0,-x^\star).
\end{array}
$$
Equivalently, we can transform  a solution $(u^\star, v^\star)\in \R^n\times \R^n$ of \eqref{prob1} into a solution $x^\star\in \R^n$ of Problem~\eqref{l2l1} by using the following transformation:
$$
x^\star=u^\star-v^\star.
$$
The Lagrangian function associated to \eqref{prob1} is
$${\cal L} (u,v,\lambda,\mu)= \frac{1}{2}\|A(u-v)-b\|^2 +\tau \sum_{i=1}^n (u_i+v_i) -\lambda^\top u - \mu^\top v,$$
with $\lambda, \mu \in \R^n$ vectors of Lagrangian multipliers.
Let $(u^\star,v^\star,\lambda^\star,\mu^\star)$ be an optimal solution of Problem~(\ref{prob1}).
Then, from necessary optimality conditions, we have
\begin{equation} \label{lambdamu}
\begin{array}{l}
\lambda^\star_i = g_i(u^\star-v^\star) +\tau\,  =  g_i(x^\star) +\tau\,;\\
\\
\mu^\star_i =  \tau\, -g_i(u^\star-v^\star) = \tau\, -g_i(x^\star).
\end{array}
\end{equation}
\par\meno
From \eqref{lambdamu}, we can introduce the following two multiplier functions
\begin{equation} \label{lambdamumult}
\begin{array}{l}
\lambda_i(u,v) = g_i(u-v) +\tau\,;\\
\\
\mu_i(u,v) =  \tau\, -g_i(u-v).
\end{array}
\end{equation}
By means of the multiplier functions, we can recall the  non-active set estimate ${\cal  N}(u,v)$ and active set estimate ${\cal  A}(u,v)$
proposed in the field of constrained smooth optimization  (see \cite{facchineilucidi} and references therein):
\begin{equation} \label{nas1}
{\cal  N}(u,v)= \{i: u_i>\epsilon\,\lambda_i(u,v)\}\cup\{i: v_i>\epsilon\,\mu_i(u,v)\},
\end{equation}
\begin{equation}\label{as1}
 {\cal A}(u,v)=I\setminus {\cal N}(u,v),
\end{equation}
where $\epsilon$ is a positive scalar.

We draw inspiration from \eqref{nas1} and \eqref{as1} to propose the new estimates of active and non-active set for Problem~\eqref{l2l1}. Indeed, by using the  relations
$$
\begin{array}{l}
u=\max(0,x)\qquad \mbox{and} \qquad v=\max(0,-x),
\end{array}
$$
we can give the following definitions.
\begin{definition}\label{def:est}
Let $x\in \R^n$. We define the following sets as estimate of the non-active and active variables sets:
\begin{equation}\label{eq:N(x^*)}
{\cal N}(x)=\{i: \max(0,x_i)>\epsilon\,(\tau+g_i(x)) \} \cup \{i: \max(0,-x_i)>\epsilon\,(\tau-g_i(x)) \},
\end{equation}
\begin{equation}
{\cal A}(x)=I\setminus {\cal N}(x).
\end{equation}
\end{definition}
In the next Subsections, we first discuss local and global properties of our estimate, then we 
compare it with other active set estimates.

\subsection{Local properties of the active set estimate}\label{subslocalpropr}
Now, we describe some local properties (in the sense that those properties only hold into a neighborhood of a given point)
of our active set estimate. In particular, the following theoretical result states that when the point is sufficiently close to an optimal solution
the related active set estimate is a subset of the active set calculated in the optimal point (and it includes the optimal active variables that satisfy strict complementarity). 
Furthermore, when strict complementarity holds the active set estimate is actually equal to the optimal active set. 
\begin{theorem}\label{activeest}
 Let $x^\star\in \R^n$ be an optimal solution of Problem~\eqref{l2l1}. Then, there
exists a neighborhood of $x^\star$ such that, for each $x$ in this neighborhood, we have
\begin{equation}\label{stcmp}
{\cal \bar A}^+(x^\star)\subseteq{\cal A} (x) \subseteq{\cal \bar A}(x^\star),
\end{equation}
with
${\cal \bar A}^+(x^\star)={\cal \bar A}(x^\star)\cap \{i : -\tau<g_i(x^\star)<\tau\}$.\\
Furthermore, if strict complementarity  \eqref{stc} holds in $x^\star$, then there
exists a neighborhood of $x^\star$ such that, for each $x$ in this neighborhood, we have
\begin{equation}\label{stcmp2}
{\cal A} (x) ={\cal \bar A}(x^\star).
\end{equation}
\end{theorem}
\begin{proof}
The proof follows from  Theorem 2.1  in \cite{facchineilucidi}.
\hfill
\end{proof}
\par\medskip

\subsection{A global property of the active set estimate}\label{subsglobpropr}
Here, we analyze a global property of the active set estimate. In particular, we show that, for a suitably chosen value of the 
parameter $\eps$ appearing in Definition~\ref{def:est}, by starting from a point $z \in \R^n$ and fixing to zero all variables whose indices belong to the active set estimate ${\cal A}(z)$,
it is possible to obtain a significant decrease of the objective function. This property, which strongly depends on the specific structure of the problem under analysis, represents a new interesting 
theoretical result, since it  enables to easily embed the active set estimate into any globally 
converging algorithmic framework (in the next section, we will show how to include it into a specific Block Coordinate Descent method). 
Furthermore, the global property cannot be deduced from the theoretical results already reported in \cite{facchineilucidi}.\\

\begin{ass}\label{ass1}
Parameter $\eps$ appearing in Definition~\ref{def:est} satisfies the following condition:
\begin{equation}\label{eps-prop2} 0 < \epsilon <\frac{1}{\lambda_{max}(A^\top A)}.\end{equation}
\end{ass}

\begin{proposition}\label{prop1}
Let Assumption \ref{ass1} hold. Given a point $z \in \R^n$ and the related sets ${\cal A}(z)$ and ${\cal N}(z)$, 
let $y$ be the point defined as
\begin{eqnarray*}
&&y_{{\cal A}(z)} = 0, \qquad \quad y_{{\cal N}(z)} = z_{{\cal N}(z)}.
\end{eqnarray*}
 Then,
$$
 f(y)-f(z)\leq - \frac{1}{2\eps}\|y-z\|^2.
$$
\end{proposition}
\begin{proof}
see Appendix \ref{appmainresact}.\hfill
\end{proof}
\par\medskip

\subsection{Comparison with other active set strategies}\label{subsas1}
Our active set estimate is somehow related to those proposed respectively by Byrd et al. in \cite{nocedal} 
and by Yuan et al. in \cite{yuan}. It is also connected in some way to the IST Algorithm (ISTA), see e.g. ~\cite{FISTA,Daubechies}.
Indeed, an ISTA step can be seen as a simple way to set to zero the variables in the context of $\ell_1$-regularized least-squares problems.

Here, we would like to point out the
similarities and the differences between those strategies and the one we propose in the present paper.

First of all, we notice that, at a generic iteration $k$ of a given algorithm, if $x^k$ is the related iterate and $i \in I$  is an index estimated 
active by our estimate, that is,
$$i \in {\cal A}(x^k)  = \{i: \max(0,x^k_i)\le \epsilon\,(\tau+g_i(x^k)) \} \cap \{i: \max(0,-x^k_i)\le\epsilon\,(\tau-g_i(x^k)) \},$$
this is equivalent to write 
\begin{equation}\label{dualcondref}
x^k_i\in [\eps(g_i(x^k)-\tau), \eps(g_i(x^k)+\tau) ] \quad\mbox{and} \quad -\tau\leq g_i(x^k)\leq\tau, 
\end{equation}
which means that $x^k_i$ is sufficiently small and satisfies the optimality condition associated with a zero component (see \eqref{optcondorpr}).
As we will see, the estimate, due to the way it is defined, tends to be more conservative than other active set  strategies (i.e. it might set to zero 
slightly smaller sets of variables). On the other hand, the global property analyzed in the previous section (i.e. decrease of the objective function when setting to zero the active  
variables) seems to indicate that the estimate truly contains indices related to variables that will be active in the optimal solution.
As we will see later on, this important property does not hold when considering the other active set strategies analyzed here.

In the block active set algorithm for quadratic $\ell_1$-regularized problems proposed in \cite{nocedal}, the active set estimate,
at a generic iteration $k$, can be rewritten in the following way:
\begin{equation}\label{stima:byrd}\nonumber
{\cal A}_{Byrd}^k = \{i\,:\, x_i^k = 0;\, g_i(x^k)\in (-\tau,\tau)\} \cup \{i\,:\, x_i^k < 0;\,  g_i(x^k) = -\tau \}
\cup \{i\,:\, x_i^k > 0;\,  g_i(x^k) = \tau \}.
\end{equation}
Let $x^k\in \R^n$ and $i\in\{1,\ldots, n\}$ be an index estimated active by our estimate, from \eqref{dualcondref}, we get $ g_i(x^k)\in[-\tau,\tau]$.

Then, in the case $x_i^k=0$,  $i\in {\cal A}_{Byrd}^k$ implies $i\in {\cal A}(x^k)$. 
In fact, let $i\in {\cal A}_{Byrd}^k$. If $x_i^k=0$ we have $g_i(x^k)\in (-\tau,\tau)$ so that $i\in {\cal A}(x^k)$.
It is easy to see that the other way around is not true.

Other differences between the two estimates come out when considering indices $i$ such that $x_i^k \ne 0$.
Let $i\in {\cal A}_{Byrd}^k$ and, in particular, $i\in \{i\,:\, x_i^k < 0;\,  g_i(x^k) = -\tau \}$.
If $|x_i^k| > \epsilon\, 2\tau$, then we get   
$$max(0, -x_i^k) = - x_i^k > \epsilon\, 2\tau = \epsilon\,(\tau - g_i(x^k)),$$
so that $i\not\in {\cal A}(x^k)$.
Using the same reasoning we can see that, in the case $i\in {\cal A}_{Byrd}^k$ and, in particular, $i\in \{i\,:\, x_i^k > 0;\,  g_i(x^k) = \tau \}$,
it can happen
$$max(0, x_i^k) = x_i^k > \epsilon\, 2\tau = \epsilon\,(\tau + g_i(x^k)),$$
so that  $i\not\in {\cal A}(x^k)$.\\

In \cite{yuan}, the active set estimate is defined as follows
\begin{equation}\label{stima:yuan}
{\cal A}_{Yuan}^k = \big\{i\,:\, x_i^k = 0;\, g_i(x^k)\in (-\tau+M^{k-1},\tau-M^{k-1})\big\},
\end{equation}
where $M^{k-1}$ is a positive scalar that measures the violation of the optimality conditions. It is easy to see that our active
set contains the one proposed in \cite{yuan}. Furthermore, we have that variables contained in our estimate are not necessarily
contained in the estimate \eqref{stima:yuan}. In particular, a big difference between our estimate  and the one proposed in
\cite{yuan} is that we can also include  variables that are non-zero at the current iterate.
\par\smallskip
 As a final comparison, we would like to point out the
differences between  the ISTA strategy and our estimate. Consider the generic iteration of ISTA with the same $\epsilon$ used in our
active set strategy:
\begin{equation}\label{one}
x^{k+1}= \arg\min_{x} \Bigg\{ q(x^k) + g(x^k)^\top (x-x^k) + \epsilon \|x - x^k\|^2  +\tau \|x\|_1 \Bigg\}.
\end{equation}
From the optimality conditions of the inner problem in \eqref{one}, we have that the zero variables at $x^{k+1}$ belong
to the following set:
\begin{equation}\label{Aista}
{\cal A}^k_{ISTA} = \{ i\, :\, \epsilon (-\tau + g_i (x^k))\le x_i^k \le \epsilon (\tau + g_i (x^k) ) \}.
\end{equation}
We can easily see that  ${\cal A}(x^k)\subseteq {\cal A}^k_{ISTA}$. The opposite is not always true, apart from the variables $x_i^k=0$. As a
matter of fact, let us consider $x_i^k>0$ and $i\in {\cal A}^k_{ISTA}$. Then, we have that
$$
\begin{array}{l l l}
 x_i^k \le \epsilon (\tau + g_i(x^k)) & \Rightarrow & i\in \{i: \ \max(0,x_i^k)\leq \eps (\tau+g_i(x^k))\}\\\\
 x_i^k \ge \epsilon (-\tau + g_i(x^k))& \Rightarrow & -x_i^k \le \epsilon (\tau - g_i(x^k))
\end{array}
$$
In order to have $i\in {\cal A}(x^k)$ it should be
$$\epsilon (\tau - g_i) \ge \max\{0, -x_i^k\} = 0$$
that is a tighter requirement with respect to the one within ${\cal A}^k_{ISTA}$. A similar reasoning applies also to variables $x_i^k<0$
with $i\in {\cal A}^k_{ISTA}$.
We would  also like to notice that  the ISTA step might generate unnecessary projections of variables to zero, thus being not
always effective as a tool for identifying the active set.

In this final remark, we show that, when using the active set strategies analyzed above, a sufficient decrease of the objective function cannot be guaranteed by setting to zero the variables in 
the active set (i.e. Proposition~\ref{prop1} does not hold). This fact makes hard, in some cases, to include those active set strategies into a globally convergent algorithmic framework.
\begin{remark}\label{remark:ASproperty}
Proposition~\ref{prop1} does not hold for the active set strategies described above. This can be easily seen in the following case.

Let us assume that, at some iteration $k$, it exists only one index $\hat \imath\in {\cal A}^k_{Byrd}$, with $x^k_{\hat \imath}>0$, 
$H_{\hat \imath \hat \imath}>0$ and 
$g_{\hat \imath}(x^k) = \tau$. Let $z=x^k$ and $y$ be the point defined as $y_i = x_i^k$ for all $i\neq \hat\imath$, and $y_{\hat\imath}=0$.
Then, 
$$f(y) = f(x^k) + (g_{\hat\imath}(x^k) - \tau) (y_{\hat\imath} - x_{\hat\imath}^k) + \frac 1 2 (y_{\hat\imath} - x^k_{\hat\imath})^2 H_{\hat\imath\hat\imath}.$$
Since $H_{\hat\imath \hat\imath}> 0 $ and  $g_{\hat\imath}(x^k) = \tau$, we have 
$f(y) - f(x^k)>0$, so that by setting to zero the active variable we get an increase of the objective function value.

The same reasoning applies also to the ISTA step, assuming that at some iteration $k$, there exists only one index 
$\hat \imath$ such that
$$\epsilon(-\tau + g_{\hat \imath}(x^k))<x^k_{\hat \imath}< \epsilon(\tau + g_{\hat \imath}(x^k))$$
and $g_{\hat \imath}(x^k) = \tau$.

Finally, it is easy to notice that, at each iteration $k$,  the active set estimate ${\cal A}^k_{Yuan}$ defined in \cite{yuan} only keeps fixed to zero, at iteration $k$, 
some of the variables that are already zero in $x^k$, thus not changing the objective function value.
\end{remark}
\par\smallskip

\section{A Fast Active Set Block Coordinate Descent Algorithm}\label{algorithm}
In this section, we describe our Fast Active SeT Block Coordinate Descent Algorithm (\texttt{FAST-BCDA}) and analyze its theoretical properties.
The main idea behind the algorithm is that of exploiting as much as possible the good properties of our active set estimate, more specifically:
\begin{itemize}
 \item[-] the ability to identify, for $k$ sufficiently large,   the ``strong'' active variables (namely, those variables satisfying the strict complementarity, see
 Theorem~\ref{activeest});
 \item[-] the ability to obtain, at each iteration, a sufficient decrease of the objective function, by fixing to zero those variables belonging to the active set estimate
 (see Proposition~\ref{prop1} of the previous section).
\end{itemize}

As we have seen in the previous section, the estimate, due to the way it is defined, tends to be more conservative than other active set  strategies (i.e. it might set to zero a
slightly smaller set of variables at each iteration). Anyway, since for each block we exactly solve an $\ell_1$-regularized subproblem, we can eventually force to zero 
some other variables in the non-active set. Another important consequence of including the $\ell_1$-norm in the subproblems  
is that we do not need any sign identification strategy for the non-active variables.

At each iteration $k$, the algorithm defines two sets
${\cal N}^k={\cal N}(x^k)$, ${\cal A}^k={\cal A}(x^k)$ and executes two steps:
\begin{itemize}
 \item[1)]  it sets to zero all of the active variables;
 \item[2)]it  minimizes only over a subset of the non-active variables, i.e. those which violate the optimality conditions the most. 
\end{itemize}
More specifically, we consider the measure related to the violation of the optimality conditions reported in \eqref{tsengest}.
 We then sort in decreasing order the indices of non-active variables 
(i.e. the set of indices ${\cal N}^k$) with respect to this measure
 and define the subset $\bar {\cal N}^k_{ord}\subseteq {\cal N}^k$ containing the first $s$ sorted indices. \par\smallskip\noindent
 The set $\bar {\cal N}^k_{ord}$ is then partitioned into $q$ subsets $I_1,\ldots, I_q$ of cardinality
 $r$, such that $s=qr$.
 Then the algorithm performs $q$ subiterations.
 At the $j$-th subiteration  the algorithm considers the set $I_j \subseteq \bar {\cal N}^k_{ord}$ and  solves to
optimality the subproblem we get from
\eqref{l2l1}, by fixing all the variables but the ones whose indices belong to $I_j$.
Below we report the scheme of the proposed algorithm (see Algorithm \ref{fig:FAST-CDA}).
\begin{algorithm}                      
\caption{\small\tt Fast Active SeT Block Coordinate Descent Algorithm (FAST-BCDA)}          
\label{fig:FAST-CDA} 
\begin{algorithmic} 
\par\vspace*{0.1cm}
\item $1$\hspace*{0.5truecm} {\bf Choose} $x^0\in \R^n$, {\bf Set} $k=0$.
\item$2$\hspace*{0.5truecm} {\bf For } $k=0, 1\ldots$
\item$3$\hspace*{1.0truecm} {\bf Compute} ${\cal A}^k$,  ${\cal N}^k$, $\bar {\cal N}^k_{ord}$ ;
\item$4$\hspace*{1.0truecm} {\bf Set} $y^{0,k}_{{\cal A}^k} = 0$ and $y^{0,k}_{{\cal N}^k} = x^k_{{\cal N}^k}$;
\item$5$\hspace*{1.0truecm} {\bf For } $j=1,\dots,q$
\item$6$\hspace*{1.5truecm} {\bf Compute } $y_{I_j}^{j,k}$, with $I_j \subseteq \bar {\cal N}^k_{ord}$, solution of problem
$$
\min_{w \in R^r} g_{I_j}(y^{j-1,k})^\top(w-y_{I_j}^{j-1,k})+\frac{1}{2} (w-y_{I_j}^{j-1,k})^\top H_{I_j I_j}(w-y_{I_j}^{j-1,k})+\tau \|w\|_1
$$
\item$7$\hspace*{1.5truecm} {\bf Set} $y_i^{j,k} = y_i^{j-1,k}$ if $i\not\in I_j$;
\item$8$\hspace*{0.8truecm} {\bf End For }
\item$9$\hspace*{0.8truecm} {\bf Set }  $x^{k+1}=y^{q,k}$;
\item $10$\hspace*{0.4truecm} {\bf End For }
\par\vspace*{0.1cm}
\end{algorithmic}
\end{algorithm}
\par\medskip
The convergence of {\tt FAST-BCDA} is based on two important results.
The first one is Proposition~\ref{prop1}, which guarantees a sufficient decrease of 
the objective function by setting to zero the variables in the active set.
The second one is reported in the proposition below. It  shows that, despite the presence of the  nonsmooth term,  by exactly
 minimizing Problem~\eqref{l2l1} with respect to  a subset $J$ of the variables (keeping all the other variables fixed),
 it is possible to get a sufficient decrease of the objective function 
 in case $\lambda_{min}(H_{JJ})>0$. 
\begin{proposition}\label{lemma2}
Given a point $z\in \R^n$ and a set  $J\subseteq I$, let
$w^*\in \R^{|J|}$ be the solution of Problem~\eqref{l2l1}, where
all variables but the ones whose indices belong to $J$ are fixed
to $z_{I\setminus J}$. Let $y\in \R^n$ be defined as
$$y_J= w^*, \quad \quad y_{I\setminus J} = z_{I\setminus J}.$$

Then we have
\begin{equation}\label{eqlemma2ver2}
  f(y)-f(z)\leq - \frac{1}{2} \lambda_{min}(H_{JJ})\|y-z\|^2.
\end{equation}
\end{proposition}
\par\smallskip
\begin{proof}
See Appendix \ref{appcon}.\hfill 
\end{proof}
\par\medskip

Now, we introduce an assumption that will enable us to prove global convergence of our algorithm. 
\begin{ass}\label{ass2}
The matrix $A\in \R^{m\times n}$ satisfies the following condition
\begin{eqnarray}\label{assunzione}
&&\min_{J} \lambda_{min}( (A^\top A)_{JJ})\ge \sigma>0,
\end{eqnarray}
where $J$ is any subset of $\{ 1,\ldots, n\}$ such that $|J|=r$, with $r$ cardinality of the blocks used in {\tt FAST-BCDA}.
\end{ass}
\par\medskip
\begin{remark} We notice that even though there are some similarities between 
Condition \eqref{assunzione} and the well-known Restricted Isometry Property (RIP) condition
with fixed order $r$  (see e.g. \cite{RIP} for further details), Condition \eqref{assunzione} is weaker than the RIP condition. 
\end{remark}
\par\medskip
Finally, we are ready to state the main result concerning the global convergence of {\tt FAST-BCDA}.
\begin{theorem}\label{teorema1}
Let Assumption \ref{ass1} and Assumption \ref{ass2} hold.  Let $\{x^k\}$ be the sequence produced by Algorithm  \texttt{FAST-BCDA}.

Then, either an integer $\bar k\geq 0$
 exists such that $ x ^{\bar k}$ is an optimal solution for Problem~\eqref{l2l1},
 or the sequence $\{x^k\}$ is infinite
and every limit point $x^\star$
 of the sequence is an optimal point for Problem~\eqref{l2l1}.
 \end{theorem}
\par\smallskip
\begin{proof}
see Appendix \ref{appcon}.\hfill 
\end{proof}
\par\medskip

Now, we discuss Assumptions \ref{ass1} and \ref{ass2} that are needed to guarantee convergence of {\tt FAST-BCDA}.
\subsection{Comments on the assumptions}\label{comm_ass}
Assumption \ref{ass1} requires the evaluation of $\lambda_{max}(A^\top A)$, which is not always easily computable for large scale problems.
Hence, we describe an updating rule for the parameter
$\eps$, that enables to avoid any ``a priori'' assumption on $\eps$. 

In practice, at each iteration $k$  we need to find the smallest
$h\in \N$ such that the value $\eps=\theta^h \tilde\eps$ and the corresponding sets ${\cal A}^k$, ${\cal N}^k$
give a point 
$$y^{0,k}_{{\cal A}^k} = 0\quad \mbox{and} \quad y^{0,k}_{{\cal N}^k} = x^k_{{\cal
N}^k}$$  satisfying
\begin{equation}
\label{eps-alg}f(y^{0,k})\le f(x^k)-\gamma\|y^{0,k}-x^k\|^2,
\end{equation}
with $\gamma > 0$. Then, we can introduce a variation of {\tt FAST-BCDA}, namely {\tt FAST-BCDA-$\eps$},
that includes the updating rule for the parameter $\eps$ in its scheme, and prove its convergence.
\begin{theorem}
Let Assumption \ref{ass2} hold.
Let $\{x^k\}$ be the sequence produced by Algorithm  \texttt{FAST-BCDA-$\eps$}.

Then, either an integer $\bar k\geq 0$
 exists such that $ x ^{\bar k}$ is an optimal solution for Problem~\eqref{l2l1},
 or the sequence $\{x^k\}$ is infinite
and every limit point $x^\star$
 of the sequence is an optimal point for Problem~\eqref{l2l1}.
 \end{theorem}
\par\smallskip
\begin{proof} The proof follows by repeating the same arguments of the proof of Theorem \ref{teorema1} by replacing the relation
(\ref{app_prop1}) with (\ref{eps-alg}). \hfill
\end{proof}
\par\medskip
Assumption \ref{ass2}, which we need to satisfy in order to guarantee
convergence of both {\tt FAST-BCDA} and {\tt FAST-BCDA-$\eps$}, is often met in practice if we
consider blocks of 1 or 2 variables (i.e. $r$ equal to 1 or 2). Indeed, when solving blocks
of 1 variable, we need to guarantee that any column $A_j$ of
matrix $A$ is such that
$$\|A_j\|^2\ge \sigma>0.$$
This is often the case when dealing with overcomplete dictionaries for signal/image reconstruction (as the columns of matrix $A$ are usually normalized, see e.g. \cite{over}).
When using 2-dimensional blocks, we want no parallel columns in the matrix $A$. This is a quite common requirement in the context
of overcomplete dictionaries (as it corresponds to ask that mutual coherence is lower than 1, see e.g. \cite{over}).
Furthermore, the solution of 1-dimensional block subproblems can be determined in closed form by means of the well-known scalar soft-threshold
function (see e.g. \cite{FISTA,SPARSA}). Similarly, we can express in closed form  the solution of 2-dimensional block subproblems.

Summarizing, thanks to the possibility to use an updating rule for $\eps$, and due to the fact that we only use blocks of dimensions 1 or 2 in our algorithm, we have that Assumptions \ref{ass1}
and \ref{ass2} are quite reasonable in practice.
\subsection{Convergence rate analysis}\label{convrate}
Here, we report a result related to the convergence rate of \texttt{FAST-BCDA} 
with 1-dimensional blocks (namely \texttt{FAST-1CDA}). In particular, we show that  it converges at a linear rate. In order to prove the result, we make an assumption that is common 
when analyzing  the convergence rate of both algorithms  for $\ell_1$-regularized problems (see e.g. \cite{ HaleYinZhang}) and
algorithms for general problems (see e.g. \cite{ortegareinbolt}):
 
\begin{ass}\label{ass3}
 Let $\{x^k\}$ be the sequence generated by \texttt{FAST-1CDA}. We have that  
  \begin{equation}\label{convtoopt}
   \lim_{k \to \infty} x^k=x^\star,
  \end{equation}
  where $x^\star$ is an optimal point of problem \eqref{l2l1}. 
\end{ass}
\par\medskip
Now, we state the theoretical result related to the linear convergence.
\begin{theorem}\label{mainconvres}
Let Assumptions \ref{ass1}, \ref{ass2} and \ref{ass3} hold. Let $\{x^k\}$ be the sequence generated by \texttt{FAST-1CDA}. 

Then $\{f(x^k)\}$ converges at least Q-linearly to $f^\star$, where $f^\star=f(x^\star)$ . Furthermore, $\{x^k\}$ converges at least 
 R-linearly to $x^\star$.
\end{theorem}
\par\smallskip
\begin{proof}
See Appendix \ref{appconrate}.\hfill
\end{proof}
\par\medskip


\section{Numerical Results}\label{numres}
In this section, we report the numerical experiments related to \texttt{FAST-BCDA}.
We implemented our method in MATLAB, and considered four different versions of it in the experiments:
\begin{itemize}
 \item \texttt{FAST-1CDA} and \texttt{FAST-2CDA}, basic versions of \texttt{FAST-BCDA} where blocks of dimension $1$ and $2$  are respectively considered;
 \item \texttt{FAST-1CDA-E} and \texttt{FAST-2CDA-E}, ``enhanced'' versions of  \texttt{FAST-BCDA}  
 where again blocks of dimension $1$ and $2$  are respectively considered (see subsection~\ref{sec:acc} for further details).
\end{itemize}
We first analyzed the performance of these four versions of our algorithm.
Then, we compared the best one with other algorithms for $\ell_1$-regularized least squares problems.
Namely, we compared \texttt{FAST-2CDA-E} with ISTA~\cite{FISTA,Daubechies}, FISTA~\cite{FISTA},  PSSgb~\cite{schmidt}, SpaRSA~\cite{SPARSA} and FPC$\_$AS~\cite{Wenetal1}.\\
All the tests were performed on an Intel Xeon(R) CPU E5-1650 v2  3.50 GHz  using MATLAB R2011b.

We considered two different testing problems of the form \eqref{l2l1}, commonly used for software benchmarking (see e.g.~\cite{Wenetal1,gondzio}).
In particular, we generated artificial signals of dimension $n= 2^{14}, 2^{15}, 2^{16}, 2^{17}$,
with a number of observations $m=n/4$ and we set the number of nonzeros  $T=round(\rho\,m)$, with $\rho=\{0.01, 0.03, 0.05, 0.07, 0.1\}$.
The two test problems (P1 and P2)  differ in the way matrix $A$ is generated:
\begin{itemize}
 \item[P1:] Considering $\bar A$ as the Gaussian matrix whose elements
are generated independently and identically distributed from the normal distribution
${\cal N} (0,1)$, the matrix $A$
was generated by scaling the columns of $\bar A$.

\item[P2:] Considering $\bar A$ as the matrix generated by using the MATLAB command
 $$A = \mbox{\texttt{sprand}(}m,n,\mbox{density)},$$
 with $\mbox{density}=0.5$,
 the matrix $A$
was generated by scaling the columns of $\bar A$.
\end{itemize}
We would like to notice that the Hessian matrices $A^\top A$ related to instances of problem P1 have most of the
mass on the diagonal. Then, those instances are in general easier to solve than the ones of problem P2.

Once the matrix $A$ was generated, the true signal $x^\star$ was built as a vector with $T$ randomly placed $\pm 1$ spikes, with zero in the other components.
Finally, for all problems, the vector of observations $b$ was chosen as $b= A\, x^\star + \eta$, where $\eta$ is a Gaussian
white noise vector, with variance $10^{-3}$. We set $\tau = 0.1 \|A^\top b\|_{\infty}$ as in \cite{salsa, SPARSA}.
We produced ten different random instances for each problem, for a total of $400$ instances.
The comparison of the overall computational effort is carried out by using the performance profiles proposed by Dolan
and Mor\'e in \cite{Dolanmore}, plotting graphs in a logarithmic scale.

For the value of $s$ (number of non-active variables to be used in $\bar {\cal N}_{ord}$) we set
$s=round(0.8\,T)$ for \texttt{FAST-1CDA} and $s=round(0.65\,T)$ for \texttt{FAST-2CDA} (these $s$ values are the ones that guarantee the
best performances among the ones we tried).
For what concerns the choice of the $\eps$ parameter used in the active set estimate, the easiest choice is that of setting
$\eps$ to a fixed value. We tested several values and obtained the best results with 
$\eps=10^{-4}$ and $\eps=10^{-5}$ for \texttt{FAST-1CDA} and \texttt{FAST-2CDA} respectively.
We further tested an implementation of both  \texttt{FAST-1CDA-$\eps$} and \texttt{FAST-2CDA-$\eps$}. Since there were no significant improvements
in the performance, we decided to keep the $\eps$ value fixed.

We would also like to spend a few words about the criterion for choosing the variables in $\bar{\cal N}_{ord}^k$. In some cases, we found more efficient  using the following measure:
\begin{equation}\label{alternative}
 \begin{array}{cc}
 |g_i (x^k)+\tau| & \mbox{if} \ x^k_i>0;\\
 |g_i (x^k)-\tau| & \mbox{if} \ x^k_i<0;\\
 \max\{0, -(g_i (x^k)+\tau), g_i(x^k)-\tau\} & \mbox{if} \ x^k_i=0,\\
\end{array}
\end{equation}
in place of the one reported in \eqref{tsengest}, which we considered for proving the theoretical results. The main feature of this new measure is that 
it only takes into account first order information (while \eqref{tsengest} considers proximity of the component value to zero too).
Anyway, replacing \eqref{tsengest} with the new measure is not a big deal, since convergence can still be proved using \eqref{alternative}.
Furthermore, linear rate can be easily obtained assuming that strict complementarity holds. Intuitively, considering only first order information in the choice of the variables should make more
sense in our context, since
proximity to zero is already taken into account when using the estimate to select the active variables.
\subsection{Enhanced version of \texttt{FAST-BCDA}}\label{sec:acc}
By running our codes, we noticed that the cardinality of the set related to the non-active variables
decreases quickly as the iterations go by. In general,
very few iterations are needed to obtain the real non-active set.
By this evidence, and keeping in mind the theoretical result reported in Section~\ref{estimate}, we decided to develop an ``enhanced'' 
version of our algorithms, taking inspiration by the second stage of FPC-AS algorithm \cite{Wenetal1}.
Once a ``good'' estimate ${\cal N}^k$ of ${\cal N}(x^\star)$ was obtained, we solved the following smooth optimization subproblem
$$
\begin{array}{l l}
\min & \frac{1}{2}\|Ax - b\|^2 + \tau sign(x_{{\cal N}^k})^\top x_{{\cal N}^k}\\
\mbox{s.t.} & x_i=0 \;\;\;\; i\in {\cal A}^k.
\end{array}
$$
In practice, we considered an estimate ${\cal N}^k$ ``good'' if both
there are no changes in the cardinality of the set with respect to the last two iterations,
and $|{\cal N}^k|$ is lower or equal than a certain threshold $\xi$ (we fixed $\xi=0.05n$ in our experiments).
%
\subsection{Preliminary experiments}\label{prel}
In order to pick the best version among the four we developed, we preliminary compared the performance of
\texttt{FAST-1CDA} (FAST1), \texttt{FAST-2CDA} (FAST2), \texttt{FAST-1CDA-E} (FAST1-E) and \texttt{FAST-2CDA-E} (FAST2-E).
In Figure~\ref{fig:PPcompprel}, we report the performance profiles with respect to the CPU time. 
\begin{figure}
\centering
\begin{subfigure}{.5\textwidth}
  \centering
   \psfrag{Fast-1CDA}[l][l]{{\tiny FAST1}}
  \psfrag{Fast-2CDA}[l][l]{{\tiny FAST2}}
  \psfrag{Fast-1CDA-acc}[l][l]{{\tiny FAST1-E}}
  \psfrag{Fast-2CDA-acc}[l][l]{{\tiny FAST2-E}}
     \includegraphics[trim = 2.0cm 0.5cm 7.5cm 0.0cm, clip, width=0.85\textwidth]{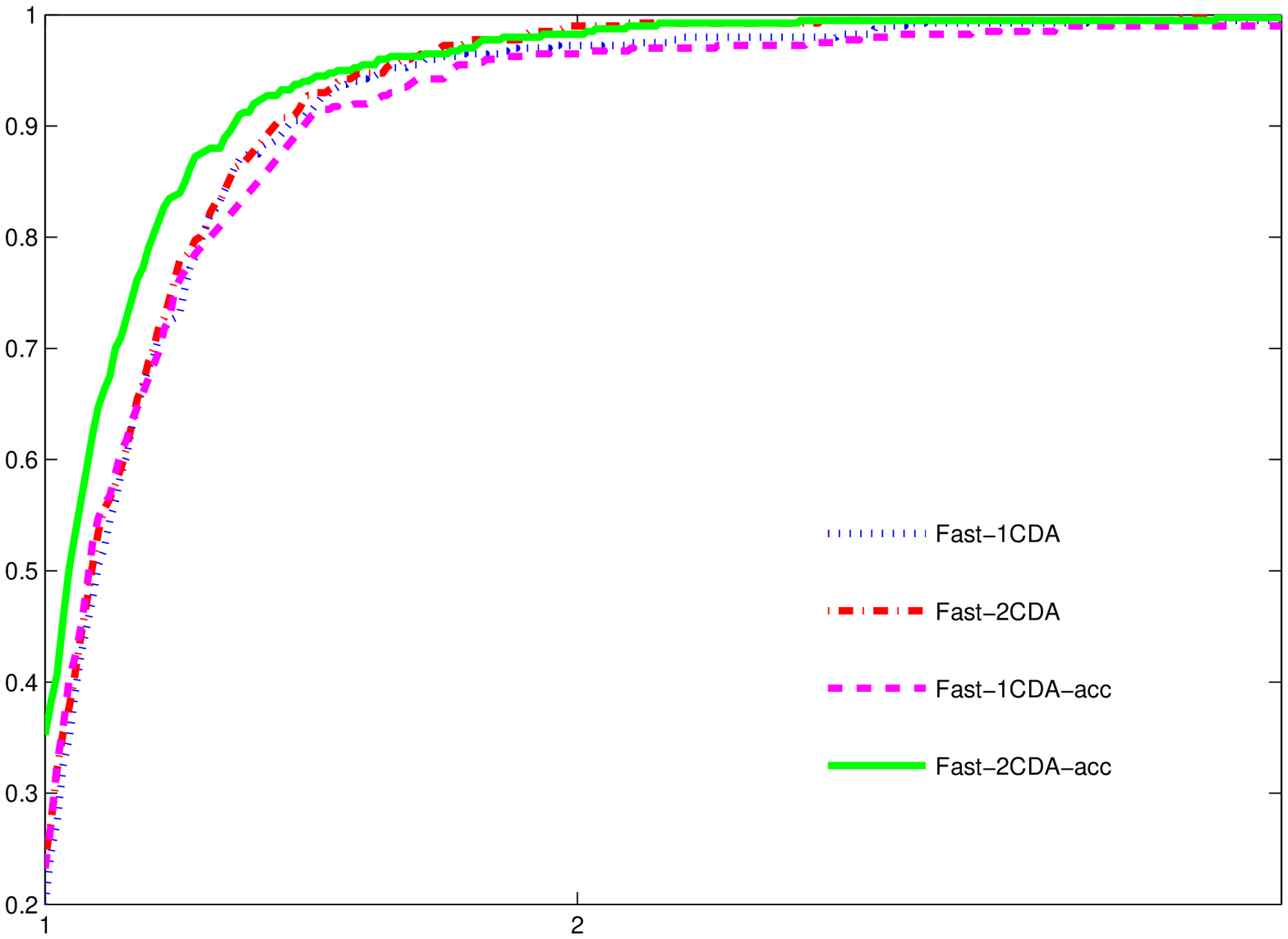}
  \caption{Preliminary experience.}
 \label{fig:PPcompprel}
\end{subfigure}%
\begin{subfigure}{.5\textwidth}
  \centering
   \psfrag{Fast-2CDA-acc}[l][l]{{\tiny FAST2-E}}
  \psfrag{ISTA}[l][l]{{\tiny ISTA}}
  \psfrag{FISTA}[l][l]{{\tiny FISTA}}
  \psfrag{PSSgb}[l][l]{{\tiny PSSgb}}
  \psfrag{SpaRSA}[l][l]{{\tiny SpaRSA}}
  \psfrag{FPC-AS}[l][l]{{\tiny FPC-AS}}
     \includegraphics[trim = 2.0cm 0.5cm 7.5cm 0.0cm, clip, width=0.85\textwidth]{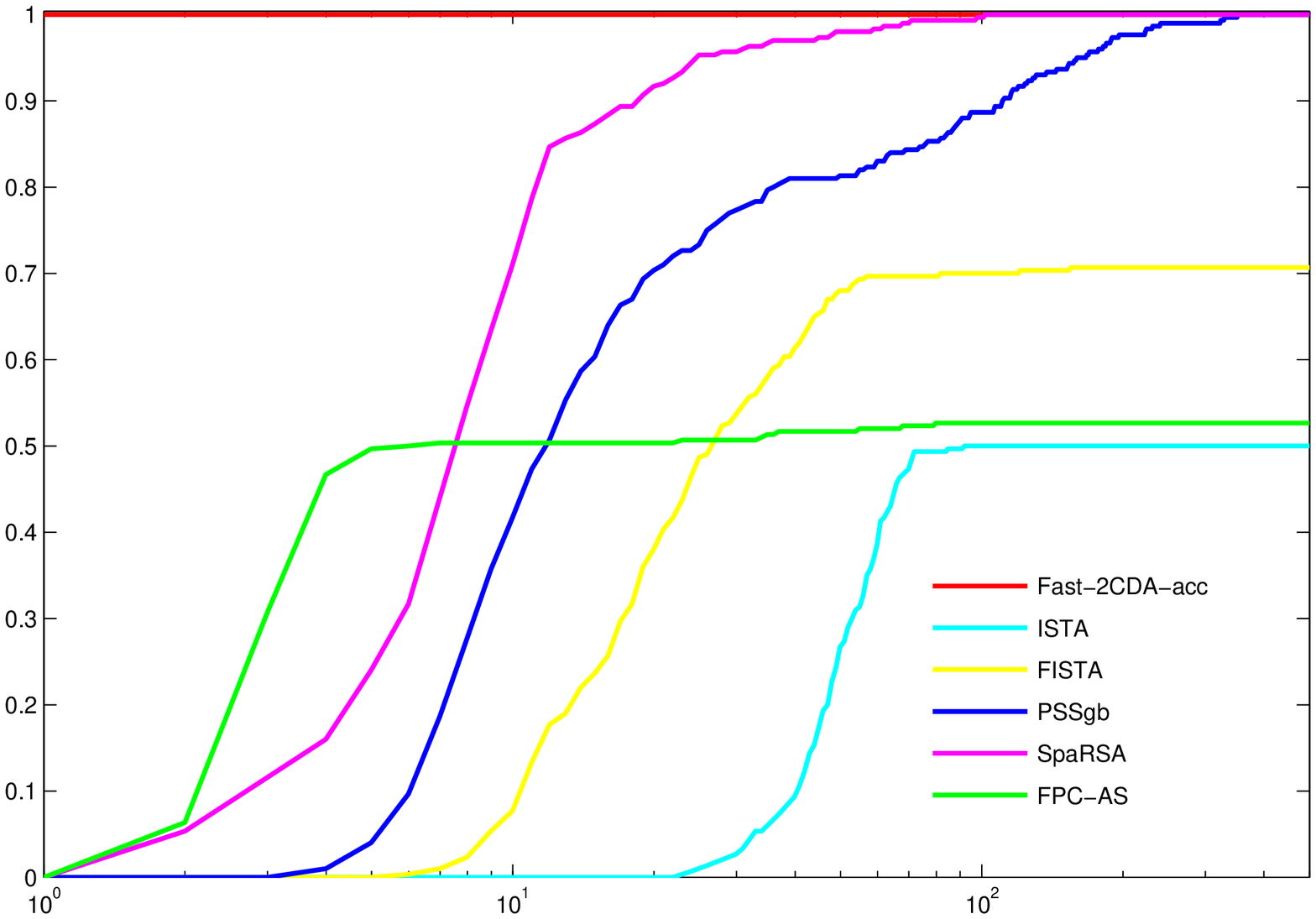}
  \caption{Comparison with other solvers.}
   \label{fig:PPmatavec}
\end{subfigure}
\caption{Performance profiles on all instances (CPU time)}
\label{fig:test}
\end{figure}

As we can see, even if the four version of \texttt{FAST-BCDA} have similar behaviour, 
\texttt{FAST-2CDA-E} is the one that gives the overall best performance. We then choose \texttt{FAST-2CDA-E} 
as the algorithm to be compared with the other state-of-the art algorithms for $\ell_1$-regularized problems.
\subsection{Comparison with other algorithms}\label{compal}
In this section, we report the numerical experience related to the comparison of \texttt{FAST-2CDA-E} 
with ISTA~\cite{FISTA,Daubechies}, FISTA~\cite{FISTA}, PSSgb~\cite{schmidt}, SpaRSA~\cite{SPARSA} and FPC$\_$AS~\cite{Wenetal1}.

In our tests, we first ran \texttt{FAST-2CDA} to obtain
a target objective function value, then ran the other algorithms until each of them reached
the given target (see e.g. \cite{SPARSA}).
Any run exceeding the limit of $1000$ iterations is considered failure.
Default values were used for all parameters in SpaRSA~\cite{SPARSA} and FPC$\_$AS~\cite{Wenetal1}.
For PSSgb~\cite{schmidt} we considered the two-metric projection method and we set the parameter \texttt{options.quadraticInit} to $1$,
since this setting can achieve better 
performance for problems where backtracking steps are required on each iteration (see \url{http://www.cs.ubc.ca/~schmidtm/Software/thesis.html}). 
In all codes, we considered the null vector as starting point and all matrices were
stored explicitly.
In Figure~\ref{fig:PPmatavec}, we report the plot of the performance profiles related to the CPU time for all instances.
From these profiles it is clear that \texttt{FAST-2CDA-E} outperforms all the other algorithms and that SpaRSA and PSSgb 
are the two best competitors.
We then further compare, in Figure~\ref{fig:boxplots-all}, \texttt{FAST-2CDA-E}, SpaRSA and PSSgb
reporting the box plots related to the distribution of the CPU time. 
On each box, the central mark
is the median, the edges of the box are the 25th and 75th percentiles, the whiskers
extend to the most extreme data points not considered outliers, and outliers are
plotted individually.
\begin{figure}[h!]
  \begin{center}
  \psfrag{Fast-2CDA-acc}[l][l]{{\tiny FAST2-E}}
  \psfrag{PSSgb}[l][l]{{\tiny PSSgb}}
  \psfrag{SpaRSA}[l][l]{{\tiny SpaRSA}}
  \psfrag{CPU time}[c][c]{{\tiny CPU time (sec)}}
  \psfrag{All instances}[c][c]{{\tiny All instances}}
  \psfrag{P1 instances}[c][c]{{\tiny P1 instances}}
  \psfrag{P2 instances}[c][c]{{\tiny P2 instances}}
     \includegraphics[trim = 2.0cm 0.5cm 2.0cm 0.0cm, clip, width=0.8\textwidth]{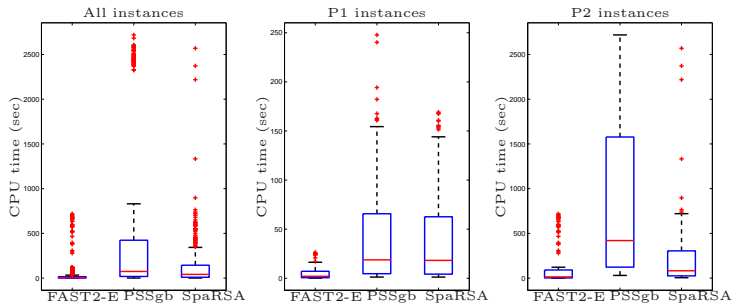}
     \caption{Box plots (CPU time).}
    \label{fig:boxplots-all}
  \end{center}
\end{figure}
In particular, Figure~\ref{fig:boxplots-all} shows the plots related to the distribution of the CPU time for all instances, for P1 instances and for P2 instances, respectively.
For what concerns P1 instances, SpaRSA and PSSgb show a similar behavior, while observing the plot related to P2 instances 
SpaRSA shows a better performance. For both classes, \texttt{FAST-2CDA-E} shows the lowest median. 
As a further comparison among \texttt{FAST-2CDA-E}, SpaRSA and PSSgb, we report in Figure~\ref{fig:errP1} and in Figure~\ref{fig:errP2}, the plots of 
the relative error vs. the CPU time for the P1 and the P2 instances respectively. In each plot, the curves are averaged over the ten runs for fixed $\rho$ and $n$.
Observing these plots, we notice that \texttt{FAST-2CDA-E} is able to reach better solutions with lower CPU time.
\subsection{Real Examples}\label{realex} 
In this subsection, we test the efficiency of our algorithm 
on realistic image reconstruction problems. We considered six images: a Shepp–Logan
phantom available through the MATLAB Image Processing Toolbox
and five widely used images downloaded from
\url{http://dsp.rice.edu/cscamera} (the letter R, the mandrill, the dice, the ball, the mug).
Each image has $128 \times 128$ pixels. 
We followed the procedure described in \cite{Wenetal1} to generate the instances (i.e. matrix $A$ and vector $b$).
%
What we want to highlight here is that the optimal solutions are unknown. Hence the reconstructed images can
only be compared by visual inspection. Also in this case, we first ran \texttt{FAST-2CDA} to obtain
a target objective function value, then ran the other algorithms until each of them reached
the given target.
The CPU-time needed 
for reconstructing the images is reported in Table \ref{tabtempi}.
In Figure~\ref{fig:realim2}, we report the images of the dice and of the mandrill 
reconstructed by \texttt{FAST-2CDA-E}, PSSgb and SpaRSA. It is interesting to notice that the quality of the reconstructed 
images can depend on the algorithm used.
In Table \ref{tabtempi}, we can easily see that \texttt{FAST-BCDA} was faster in all problems.

\begin{figure}[h!]
  \begin{center}
  \psfrag{CPU time (sec)}[c][c]{{\tiny CPU time (sec)}}
  \psfrag{relative error}[c][c]{{\tiny relative error}}
  \psfrag{Fast2-E}[l][l]{{\tiny FAST2-E}}
  \psfrag{PSSgb}[l][l]{{\tiny PSSgb}}
  \psfrag{SpaRSA}[l][l]{{\tiny SpaRSA}}
  \includegraphics[trim = 2.0cm 0.5cm 2.0cm 0.0cm, clip, width=0.79\textwidth]{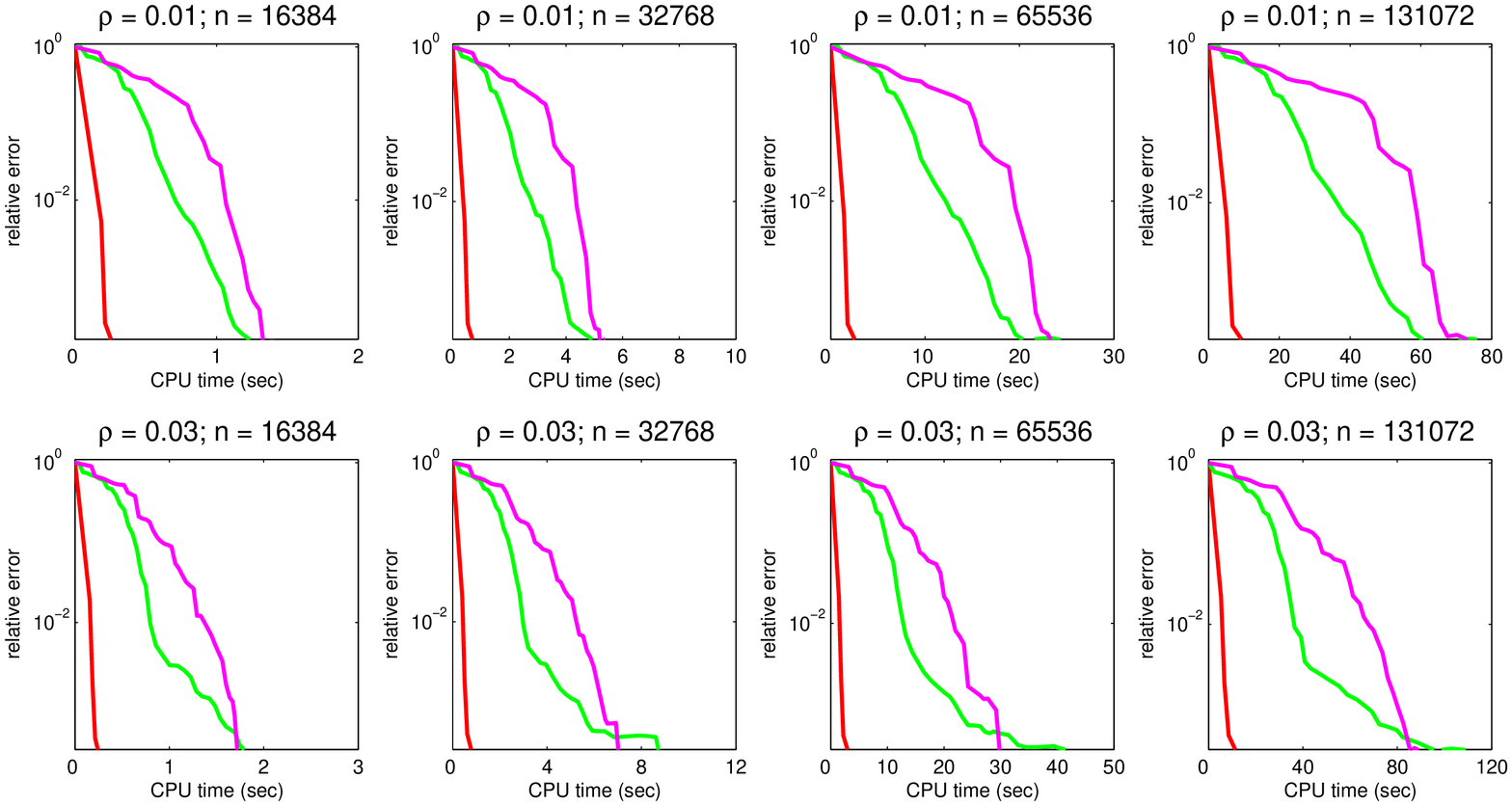}\\
  \includegraphics[trim = 2.0cm 0.5cm 2.0cm 0.0cm, clip, width=0.79\textwidth]{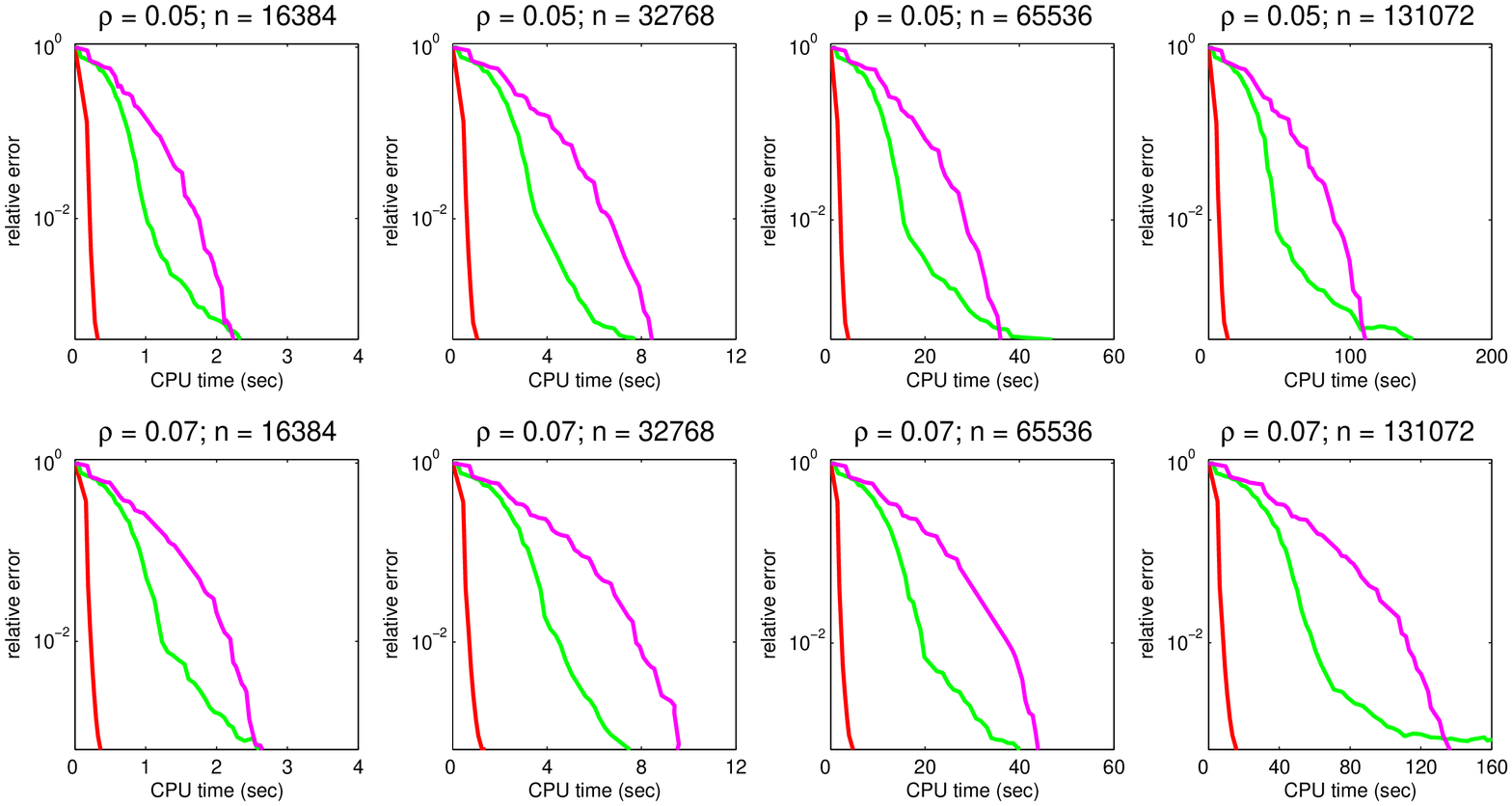}\\
  \includegraphics[trim = 2.0cm 7.5cm 2.0cm 0.0cm, clip, width=0.79\textwidth]{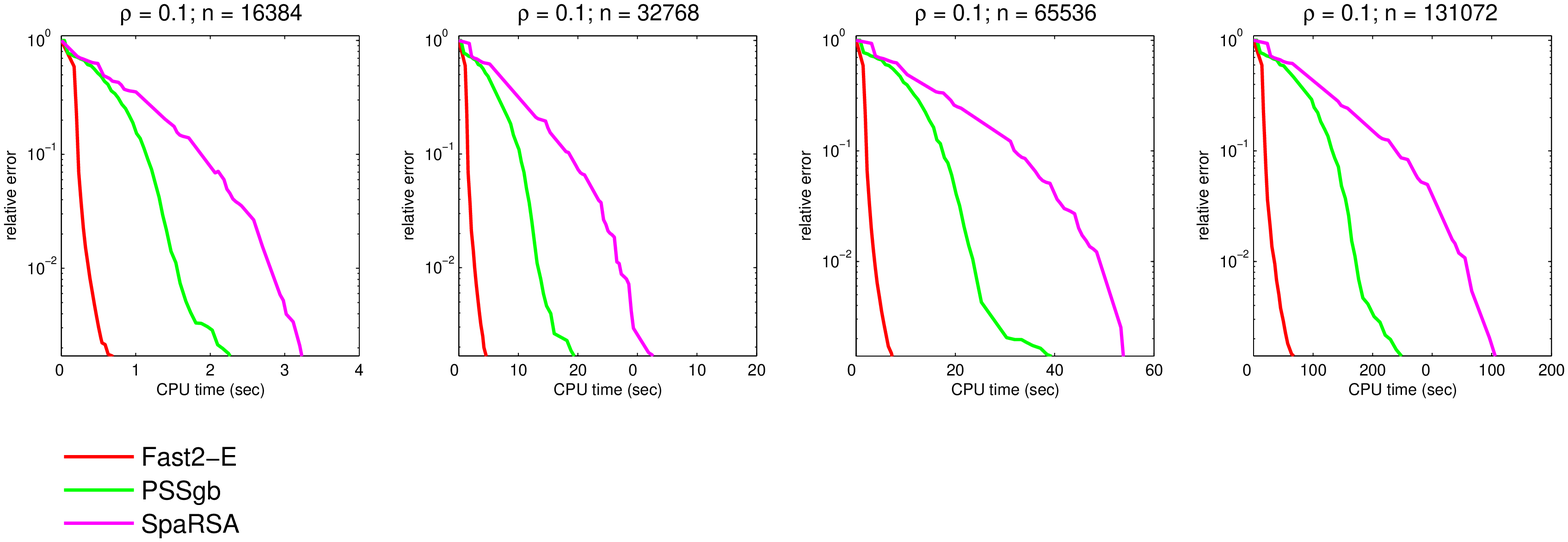}\\
     \caption{Relative error vs. CPU time -  P1 instances}
    \label{fig:errP1}
  \end{center}
\end{figure}

\begin{figure}[h!]
  \begin{center}
  \psfrag{CPU time (sec)}[c][c]{{\tiny CPU time (sec)}}
  \psfrag{relative error}[c][c]{{\tiny relative error}}
   \psfrag{FAST2-E}[l][l]{{\tiny FAST2-E}}
  \psfrag{PSSgb}[l][l]{{\tiny PSSgb}}
  \psfrag{SpaRSA}[l][l]{{\tiny SpaRSA}}
  \includegraphics[trim = 2.0cm 0.5cm 2.0cm 0.0cm, clip, width=0.79\textwidth]{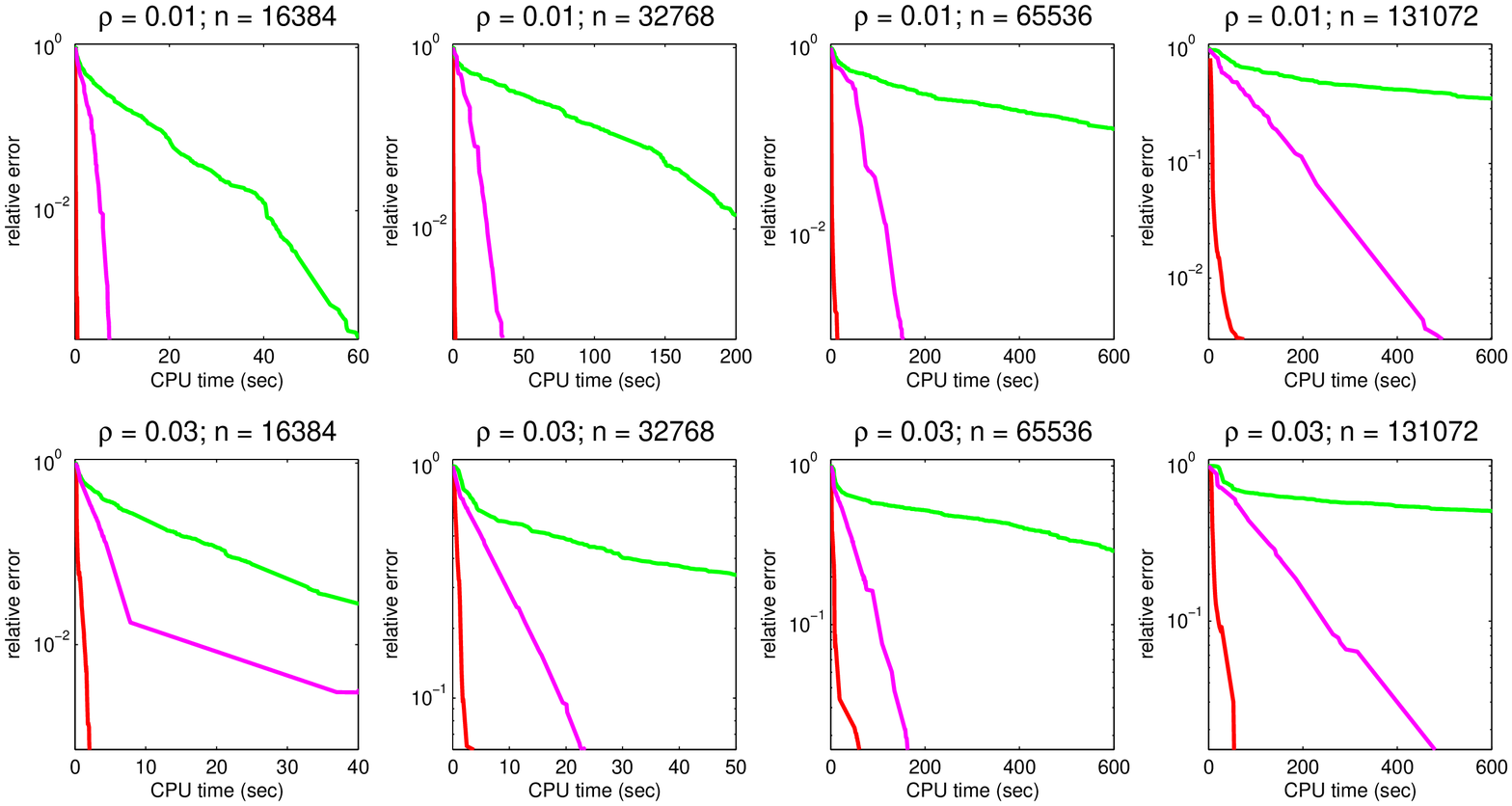}\\
  \includegraphics[trim = 2.0cm 0.5cm 2.0cm 0.0cm, clip, width=0.79\textwidth]{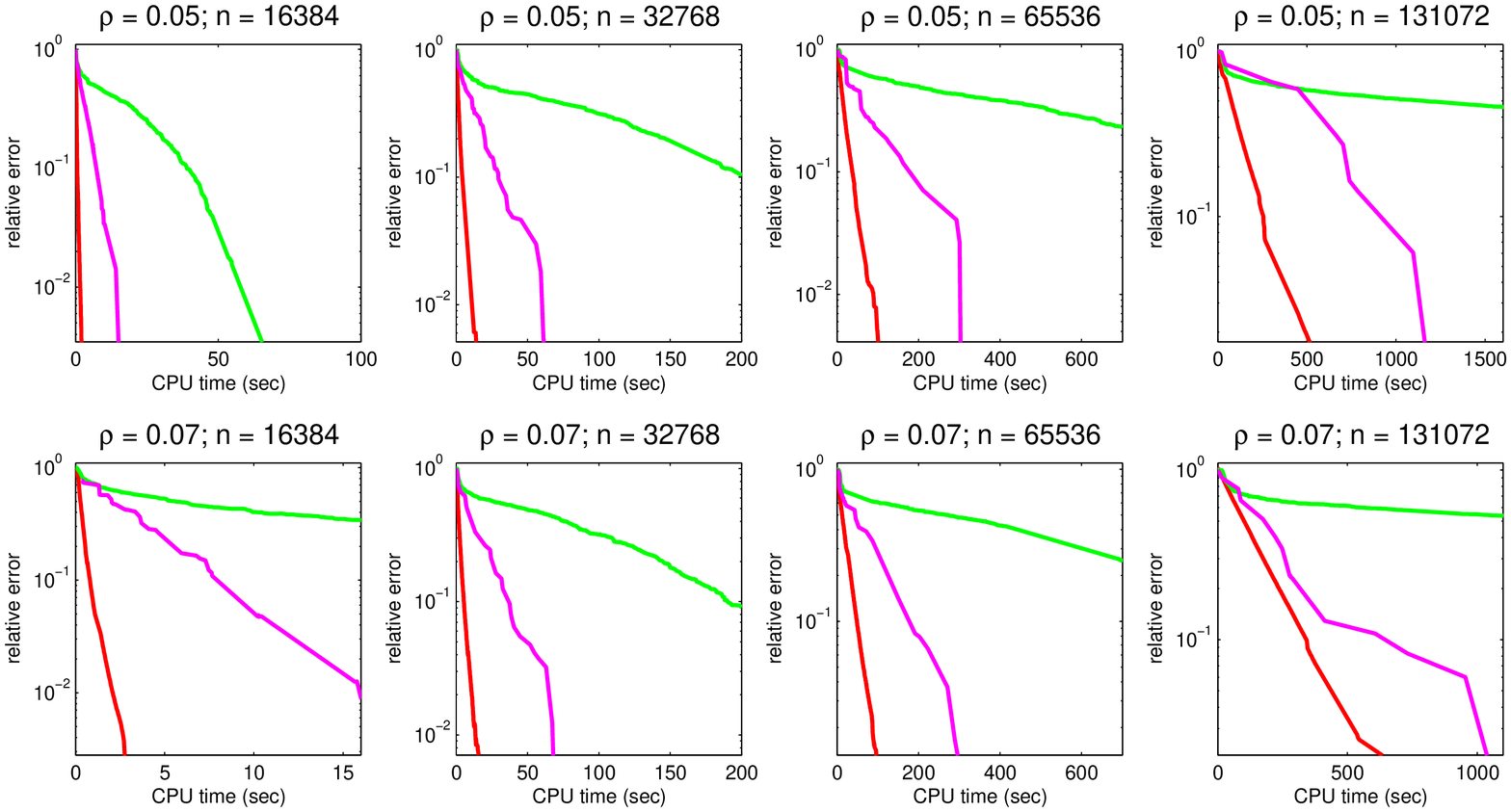}\\
  \includegraphics[trim = 2.0cm 7.5cm 2.0cm 0.0cm, clip, width=0.79\textwidth]{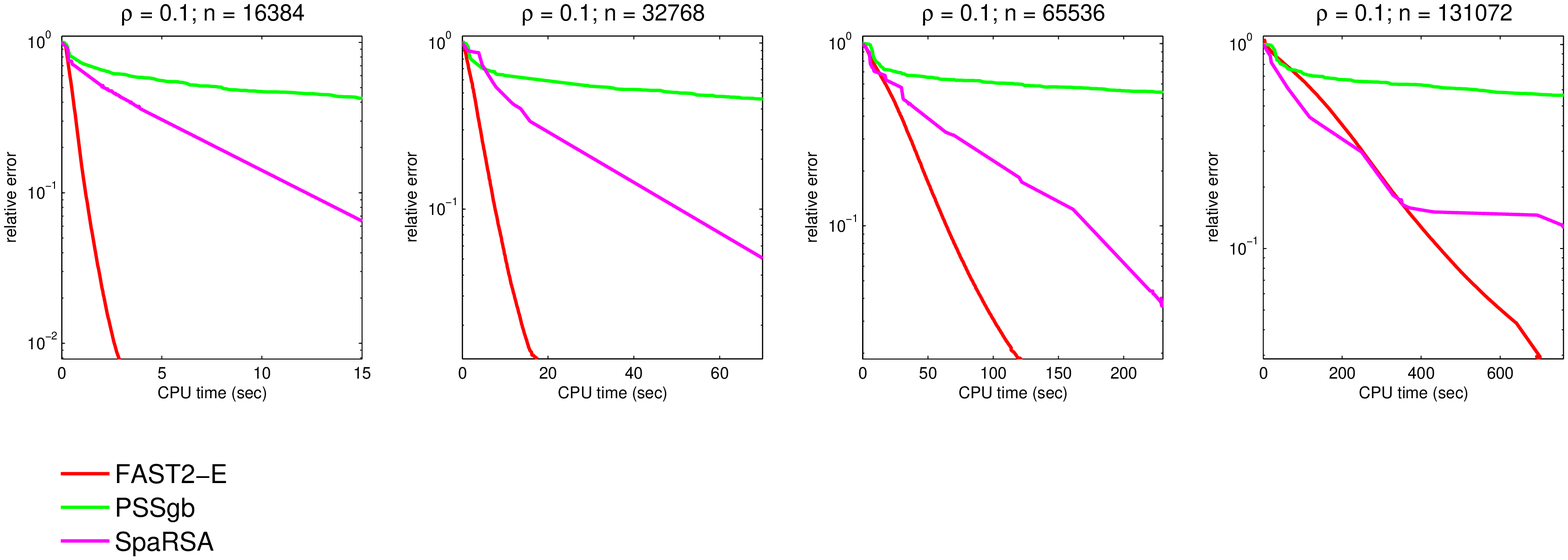}\\
     \caption{Relative error vs. CPU time -  P2 instances}
    \label{fig:errP2}
  \end{center}
\end{figure}
%

\begin{figure}[h!]
  \begin{center}
  
  \includegraphics[trim = 2.0cm 0.5cm 2.0cm 0.0cm, clip, width=8cm]{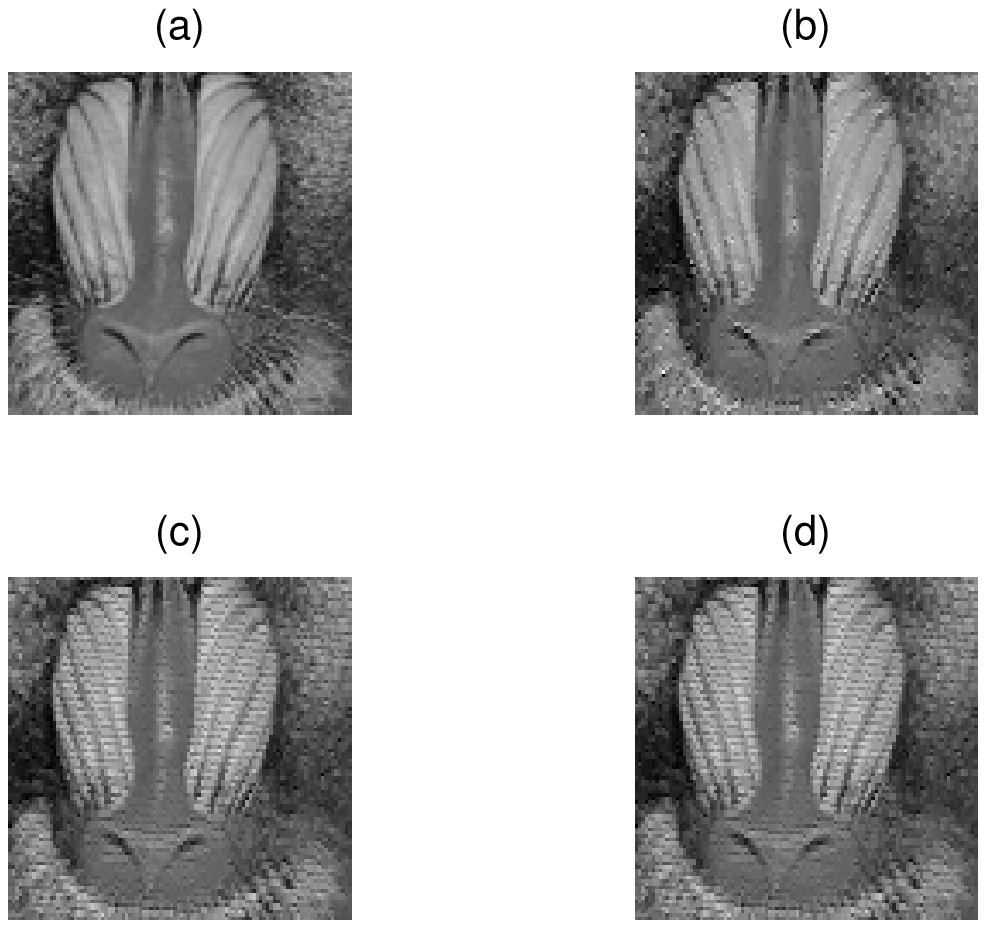}\\
  \includegraphics[trim = 2.0cm 0.5cm 2.0cm 0.0cm, clip, width=8cm]{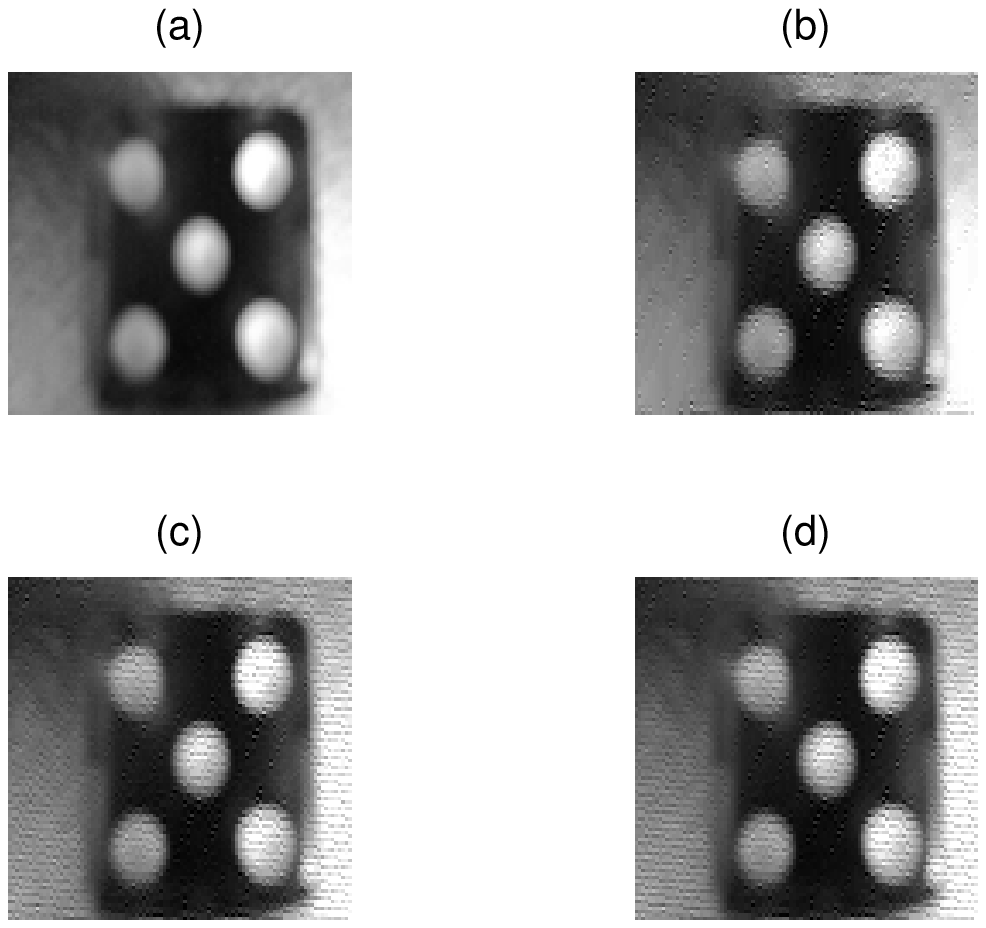}\\
     \caption{Real Examples Experiment. (a) original image - (b) \texttt{FAST-2CDA-E} reconstruction - (c) PSSgb  reconstruction - (d) SpaRSA  reconstruction}
    \label{fig:realim2}
  \end{center}
\end{figure}

      \begin{table}
        {\scriptsize
          \centering
          \begin{tabular}{|c|c|c|c|c|c|c|c|c|} 
            \hline\hline
            Fast1& Fast2& Fast1-E&Fast2-E& ISTA& FISTA& PSSgb& SpaRSA& FPC$\_$AS\\
            \hline
2.18	&	2.46		&		2.54		&		2.02		&		34.65		&		9.26		&		5.01		&		5.08		&		10.46		\\
1.65	&	1.67		&		1.51		&		1.95		&		73.01		&		16.31		&		7.77		&		14.48		&		12.87		\\
1.86	&	1.97		&		1.91		&		1.65		&		78.05		&		18.41		&		7.90		&		15.56		&		14.69		\\
3.52	&	2.05		&		2.13		&		2.32		&		63.13		&		12.69		&		6.51		&		7.53		&		9.77		\\
2.29	&	2.12		&		1.79		&		2.16		&		51.12		&		13.79		&		6.36		&		11.13		&		9.33		\\
4.12	&	4.21		&		4.16		&		2.41		&		56.66		&		12.09		&		6.69		&		7.12		&		9.72		\\

            \hline\hline
          \end{tabular}\vspace*{0.1cm}
          \caption{Real Examples Experiment - CPU time.}
          \label{tabtempi}
        }
      \end{table}

\section{Conclusions}\label{conclusions}
In this paper, we devised an active set-block coordinate descent method (\texttt{FAST-BCDA}) for solving $\ell_1$-regularized
least squares problems. The way the active set estimate is calculated guarantees a sufficient decrease 
in the objective function at every iteration when setting to zero the variables estimated active.
Furthermore, since the subproblems related to the blocks explicitly take into account the $\ell_1$-norm, the proposed algorithmic framework  
does not require a sign identification strategy for the non-active variables. 

Global convergence of the method is established. A linear convergence result is also proved. Numerical results are presented
to verify the practical efficiency of the method, and they indicate that \texttt{FAST-BCDA}
compares favorably with other state-of-the-art techniques.

We further would like to remark that the proposed active set strategy is independent
from the specific algorithm we have designed and can be easily included into other algorithms for $\ell_1$-regularized least
squares, both sequential and parallel, to improve their performance.
We  finally highlight that the algorithmic scheme we described can be easily modified in order to work
in a parallel fashion. Future work will be devoted to adapt the presented approach to handle convex $\ell_1$-regularized problems.
\par\medskip\noindent
\textbf{Acknowledgments} The authors would like to thank the Associate Editor and the anonymous Reviewers for their thorough and useful comments that significantly helped to improve the paper.
\par\medskip\noindent
\appendix
\small
\section{Main theoretical result related to the active set estimate}\label{appmainresact}
Here, we prove the main theoretical result related to the active set estimate.
\par\smallskip
\emph{Proof of Proposition \ref{prop1}.} We first define the sets  ${\cal N}={\cal N}(z)$ and ${\cal A}$=${\cal A}(z)$.
By taking into account the definitions of the sets
${\cal A}$ and ${\cal N}$  and the points  $y$ and $z$, we have:
\begin{equation}\label{funzione-ob}
f(y)= q(y)+ \tau \sum_{i=1}^n sign(y_i)\,y_i=q(y)+ \tau \sum_{i\in
{\cal N}} sign(y_i)\,y_i+ \tau \sum_{i\in {\cal A}}
sign(z_i)\,y_i.
\end{equation}
from which
\begin{eqnarray*}
&& f(y) = f(z)+ (g_{{\cal A}}(z)+\tau S_{{\cal A}}e)^\top(y-z)_{{\cal
A}}+
 \frac{1}{2}(y-z)_{{\cal A}}^{T} H_{ {\cal A}  {\cal A}} (y-z)_{ {\cal A}},
\end{eqnarray*}
where $e\in \R^{|{\cal A}|}$ is the unit vector, and   $S_{ {\cal A}}$ is the diagonal matrix defined as
$$S_{{\cal A}} = Diag(sign(z_{ {\cal A}})),$$ with the function $sign(\cdot)$ intended componentwise.\par\smallskip\noindent
Since $H=A^\top A$ we have that the following inequality holds
\begin{eqnarray}\nonumber
 &&f(y)
 \leq f(z)+ (g_{ {\cal A}}(z)+\tau S_{ {\cal A}}e)^\top(y-z)_{ {\cal A}}+\frac{\lambda_{max}(A^\top A)}{2}\|(y-z)_{ {\cal A}}\|^2.
\end{eqnarray}
Recalling (\ref{eps-prop2})  we obtain:
\begin{eqnarray}
 &&f(y)
 \leq f(z)+ (g_{ {\cal A}}(z)+\tau S_{ {\cal A}}e)^\top(y-z)_{ {\cal A}}+\frac{1}{2\eps}\|(y-z)_{ {\cal A}}\|^2.
\end{eqnarray}
\noindent
Then, we can write
$$f(y)\leq f(z)+ \Big(g_{{\cal A}}(z)+\tau S_{{\cal A}}e+ \frac{1}{\eps}(y-z)_{{\cal A}}\Big)^\top(y-z)_{ {\cal A}}-\frac{1}{2\eps}\|(y-z)_{ {\cal A}}\|^2. $$
In order to prove the proposition, we need to show that
\begin{equation}\label{eqlemma}
 \Big(g_{{\cal A}}(z)+\tau S_{{\cal A}}e+ \frac{1}{\eps}(y-z)_{{\cal A}}\Big)^\top(y-z)_{{\cal A}} \leq 0.
\end{equation}
Inequality \eqref{eqlemma} follows from the fact that $\forall i\in  {\cal A}$:
\begin{equation}\label{eqlemmai}
 \Big(g_i(z)+\tau  sign(z_i)+ \frac{1}{\eps}(y_i-z_i)\Big)^\top(y_i-z_i) \leq 0.
\end{equation}
We distinguish two cases:
\begin{itemize}
\item[a)] If $z_i > 0$, we have that  $sign(z_i)= 1$ and, since $y_i = 0$,
$(y_i-z_i)\leq 0$.

Then, from the fact that $i\in  {\cal A}$, we have
\begin{eqnarray}
 y_i&=& 0\nonumber \\
 z_i &\le& \eps\,( g_i(z) + \tau) \nonumber \\
 (z_i - y_i) &\le& \eps\,( g_i(z) + \tau) \nonumber \\
 \frac{1}{\eps}\,(z_i - y_i) &\le& g_i(z) + \tau \nonumber
\end{eqnarray}
so that
\begin{equation}
 g_i(z) + \tau +\frac{1}{\eps}\,(y_i - z_i)\geq 0. \nonumber
\end{equation}
and  \eqref{eqlemmai} is satisfied. \item[b)] If $z_i < 0$, we
have that  $sign(z_i)= - 1$ and, since $y_i = 0$, $(y_i-z_i)\geq 0$.

Then, by reasoning as in case a), from the fact that $i\in  {\cal A}$, we can write
\begin{eqnarray}
 y_i&=& 0\nonumber \\
- z_i &\le& \eps\,( \tau -g_i(z)) \nonumber \\
 (y_i - z_i) &\le& \eps\,( \tau -g_i(z)) \nonumber \\
 \frac{1}{\eps}\,(y_i - z_i) &\le&  \tau -g_i(z) \nonumber
\end{eqnarray}
from which we have:
\begin{equation}
 g_i(z) - \tau +\frac{1}{\eps}\,(y_i - z_i)\leq 0. \nonumber
\end{equation}
Again, we have that \eqref{eqlemmai} is satisfied.
\end{itemize}
\hfill$\Box$
\par\medskip

\section{Theoretical results related to the convergence analysis}\label{appcon}
First, we prove the result that guarantees a sufficent decrease when minimizing with respect to a given block. 
\par\smallskip
\emph{Proof of Proposition \eqref{lemma2}.} Let us consider the subproblem obtained by 
fixing all variables in $I$ but the ones whose indices belong to $J$ to $z_{I\setminus J}$.
Let $w^*\in \R^{|J|}$ be a solution of this subproblem.

We consider the set $J=\{j_1,\ldots,j_{|J|}\}$ as the union of two sets
$$J=J_E \cup J_D,$$
where
$$J_E=J_{E^+} \cup J_{E^-}, \qquad J_D=J_{D^+} \cup J_{D^-}$$
and
$$
\begin{array}{l}
 J_{D^+}=\{j_i\in J:  \, sign(w^*_i)>0 \}; \qquad 
 J_{D^-}=\{j_i\in J:  \, sign(w^*_i)<0 \};\\
 \\
 J_{E^+}=\{j_i\in J: w^*_i=0; \, sign(z_{j_i})>0 \}; \qquad
 J_{E^-}=\{j_i\in J: w^*_i=0; \, sign(z_{j_i})<0 \}.
\end{array}
$$
Let $\tilde f:\R^{|J|} \rightarrow \R$, with $w\in \R^{|J|}$, be the
following function:
$$
\begin{array}{l l}
\tilde f(w)  = & q(z)+ \tau \sum_{j\in {I\setminus J}} sign(z_j)\,z_j+ g_J(z)^\top(w-z_J)+\frac{1}{2} (w-z_J)^\top H_{JJ} (w-z_J) \\
\\
 & + \;\tau\sum_{j_i\in  J_E} sign(z_{j_i})\,w_i+\tau \sum_{j_i\in  J_D} sign(w^*_i)\,w_i.
 \end{array}
 $$
Then, $w^*$ can be equivalently seen as the solution of the
following problem
\begin{equation}\label{pl-ver1}
\begin{array}{ll}
\min &\tilde f(w)\\
\\
s.t.&  w_i \ge 0\quad \mbox{for } j_i\in J_{D^+}\cup J_{E^+},\\
\\
    &  w_i \le 0\quad \mbox{for } j_i\in J_{D^-}\cup J_{E^-},\\
\end{array}
\end{equation}
By introducing the diagonal matrix $S= Diag(s)\in \R^{|J|\times |J|} $,
where $s\in \{-1,0,1\}^{|J|}$ is the vector defined as
$$
s_i=\left\{
\begin{array}{l l}
sign(w^*_i) & \mbox{if } j_i\in J_D\\
\\
sign(z_{j_i}) & \mbox{if } j_i\in J_E,\\
\end{array}
\right.
$$
Problem~\eqref{pl-ver1} can be written in a more compact form as
\begin{equation}\label{pl}
\begin{array}{ll}
\min &\tilde f(w) \\
\\
 s.t.& S w \ge 0.\\
\end{array}
\end{equation}
From the KKT  condition for Problem~\eqref{pl} at $w^*$
we have:
\begin{equation}\label{KKT}
g_J(z) + H_{JJ} (w^* - z_J) + \tau s -  S \lambda= 0;
\end{equation}
where $\lambda \in \R^{|J|}$ is the vector of multipliers with respect to the constraints $S w \ge 0$.

We now analyze \eqref{KKT} for each index $i\in J$. We distinguish two cases:
\begin{itemize}
\item[-] $j_i \in J_D$. In this case we have that $s_i = sign(w^*_i)$
and $\lambda_i = 0$. Then, from \eqref{KKT} we have
\begin{equation}\label{KKTw}
g_{j_i}(z) + H_{j_i j_i} (w^*_i - z_{j_i}) + \tau s_i = 0.
\end{equation}
\item[-] $j_i\in J_E$. In this case we have that $s_i =
sign(z_{j_i})$ and $\lambda_i \ge 0$.
\end{itemize}
Therefore,
$$
g_{j_i}(z) + H_{j_i j_i} (w^*_i - z_{j_i}) + \tau s_i \geq 0  \quad \mbox{ if } s_i =sign(z_{j_i})\geq 0,\\
$$
$$
g_{j_i}(z) + H_{j_i j_i} (w^*_i - z_{j_i}) + \tau s_i \leq 0  \quad \mbox{ if } s_i =sign(z_{j_i})\leq 0.\\
$$
The previous inequalities  and the fact that $w^*_i=0$ for all
$j_i\in J_E$  imply that, whatever is the sign of $z_i$, we have
\begin{equation}\label{KKTx}
\Big(g_{j_i}(z) + H_{j_i j_i}(w^*_i - z_{j_i}) + \tau
s_i\Big)(w^*_i - z_{j_i}) \leq 0.
\end{equation}
Taking into account \eqref{KKTw} and \eqref{KKTx}, we have that
\begin{equation}\label{gradleq0}
\Big(g_J(z) + H_{JJ} (w^* - z_J) + \tau s\Big)^\top (w^* -
z_J) \leq 0.
\end{equation}
Now, consider the difference between $\tilde f(w^*)$ and  $\tilde
f(z_J)$. We have that
\begin{eqnarray*}
\tilde f(w^*) - \tilde f(z_J) &= &g_J(z)^\top(w^*-z_J)+\frac{1}{2} (w^*-z_J)^\top H_{JJ} (w^*-z_J) \\
&& +\; \tau \sum_{j_i\in  J_E} sign(z_{j_i})(w^*_i-z_{j_i})+\tau
\sum_{j_i\in  J_D} sign(w^*_i)(w^*_i-z_{j_i})_,
\end{eqnarray*}
which can be rewritten as
$$
{\small
\tilde f(w^*) - \tilde f(z_J) = \Big(g_J(z)+ H_{JJ}(w^*-z_J)
+ \tau s\Big)^\top (w^*-z_J) - \frac{1}{2} (w^*-z_J)^\top
H_{JJ} (w^*-z_J).
}
$$
Recalling \eqref{gradleq0}  and the fact that $y_J = w^*$ we have
\begin{equation}\label{new2}
\tilde f(w^*) - \tilde f(z_J)\le - \frac{1}{2} (w^*-z_J)^\top
H_{JJ} (w^*-z_J)\le  - \frac{1}{2}
\lambda_{min}(H_{JJ})\|y-z\|^2.
\end{equation}
Since
$$q(y) = q(z) + g_J(z)^\top(y-z)_J+\frac{1}{2} (y-z)^\top_J H_{JJ} (y-z)_J, $$
by definition of $\tilde f$ we have that
{\small \begin{eqnarray}\label{new3}
 f(y)=q(y) &+& \tau\sum_{j=1}^n sign(y_j)y_j = q(y)+ \tau \sum_{j\in {I\setminus J}} sign(z_j)\,z_j +\\ \nonumber
     &+&  \;\tau\sum_{j_i\in  J_E} sign(z_{j_i})\,w_i^*+\tau \sum_{j_i\in  J_D} sign(w^*_i)\,w_i^* = \tilde f(w^*)
 \end{eqnarray}}
 and
 {\small \begin{equation}\label{new4}
\tilde f(z_J)= q(z)+ \tau \sum_{j_i\in {I\setminus J_D}}
sign(z_{j_i})z_{j_i} +\tau \sum_{j_i\in  J_D}
sign(w^*_i)z_{j_i}\le q(z)+ \tau \| z\|_1 =f(z).
\end{equation}}
Now (\ref{new2}), (\ref{new3}) and (\ref{new4})
prove the Proposition.
\hfill $\Box$
\par\medskip

Then, we prove the main convergence result related to {\tt FAST-BCDA}.
\par\smallskip
\emph{Proof of Theorem \ref{teorema1}.} We first prove that {\tt FAST-BCDA} is well defined (in the sense that  $x^{k+1}\neq x^k$ iff  the point $x^k$ is not an  optimum). Let $x^k$ not be optimum, then by contradiction
we assume that $x^{k+1}= x^k$. Thus we have that either ${\cal N}(x^k)=\emptyset$, or forall $i \in {\cal A}(x^k), \ x^k_i=0$. This, in turns, implies that $x^k=0$ and, by taking into account the 
definition of ${\cal A}(x^k)$, we have that $x^k$ is optimal, thus getting a contradiction. The proof of the other implication easily follows from Propositions \ref{prop1} and \ref{lemma2}. 

Let  $\{y^{h,k}\}$, with $h=0,\ldots,q$ be the sequence of points produced by Algorithm \texttt{FAST-BCDA}.
By setting $y=y^{0,k}$ and $z=x^k$ in Proposition~\ref{prop1}, we have:
\begin{eqnarray}\label{app_prop1}
&&f(y^{0,k})\leq f(x^k) - \frac{1}{2\eps}\|y^{0,k}-x^k\|^2.
\end{eqnarray}
 By setting $y=y^{h+1,k}$ and $z=y^{h,k}$, for $h=0,\ldots,q-1$ in Proposition~\ref{lemma2}, we have:
\begin{eqnarray}\label{app_lemma2}
&&f(y^{h+1,k})\leq f(y^{h,k}) -
\frac{\sigma}{2}\|y^{h+1,k}-y^{h,k}\|^2.
\end{eqnarray}
\noindent
By using  (\ref{app_prop1}) and (\ref{app_lemma2}), we can write
\begin{eqnarray}\label{funzione-1}f(x^{k+1})\leq f(y^{q-1, k})\leq\dots\leq f(y^{0,k})\leq f(x^k),\end{eqnarray}
from which we have:
$$x^k\in {\cal L}^0=\{x\in \R^n: \ f(x)\leq f(x^0)\}.$$
 From the coercivity of the objective function of Problem~\eqref{l2l1} we have that the level set
${\cal L}^0$ is compact. Hence, the sequence $\{x^k\}$ has at least a limit point and
\begin{eqnarray}\label{funzione-2}
&&\lim_{k\to\infty}\bigl(f(x^{k+1})-f(x^{k})\bigr)=0.
\end{eqnarray}
Now, let $x^\star$  be any limit point of the sequence $\{x^k\} $ and $\{x^k\}_K $ be the subsequence such that
\begin{eqnarray}\label{sub-sequence-1}
&&\lim_{k\to\infty, k\in K}x^{k}=x^\star.
\end{eqnarray}
Let us assume, by contradiction, that $x^\star$ is not an optimal point of  Problem~\eqref{l2l1}. 
By taking into account that  inequality $\|\sum_{i=1}^l a_i\|\leq l\ \sum_{i=1}^l \|a_i\|^2$ holds for the squared norm of sums of $l$ vectors $a_i$,
and by recalling (\ref{app_prop1}), (\ref{app_lemma2}) and (\ref{funzione-1}),  we have
\begin{eqnarray}\label{funzione-3}
&& f(x^{k+1})\leq f(y^{0,k})\leq f(x^k) - \frac{1}{2\eps}\|y^{0,k}-x^k\|^2,\\
&& \label{funzione-4}
f(x^{k+1})\leq f(y^{h,k}) \leq f(x^k) - \frac{\sigma}{2}\|y^{h,k}-x^k\|^2.
\end{eqnarray}
with $h=1,\ldots,q$.

Now, (\ref{funzione-2}), (\ref{sub-sequence-1}), (\ref{funzione-3}) and (\ref{funzione-4}) imply
\begin{eqnarray}\label{convy}
\lim_{k \to \infty, k \in K} y^{h,k} = x^\star,
\end{eqnarray}
for $h=0,\ldots,q$.

For every index $j\in {\cal A}^k$, we can define  the point $\tilde y^{j,k}$ as follows:
\begin{equation}\label{def-y}
\tilde y_i^{j,k}=\left \{
\begin{array}{ll}
0 &\mbox{if} \ i=j  \\
x_i^k &\mbox{otherwise}\\
\end{array}
\right .
\end{equation}
Recalling the definition of points $\tilde y^{j,k}$ and  $y^{0,k}$, we have
$$ \|\tilde y^{j,k}-x^k\|^2=(\tilde y^{j,k}-x^k)_j^2=( y^{0,k}-x^k)_j^2\le ( y^{0,k}-x^k)_j^2+\sum_{i\in {\cal A}^k, i\not=j}(x^k_i)^2=\| y^{0,k}-x^k\|^2.$$
From the last inequality and (\ref{convy}) we obtain
\begin{eqnarray}\label{convytilde}
\lim_{k \to \infty, k \in K} \tilde y^{j,k} = x^\star,
\end{eqnarray}
for all $j\in {\cal A}^k$.

To conclude the proof, we consider the function $\Phi_i(x)$,  defined in \eqref{tsengest},  that measures the violation of the optimality conditions for a variable $x_i$.

Since, by contradiction, we assume that $x^\star$ is not an optimal point there must  exists
an index $\hat \imath$ such that
\begin{equation}\label{cont0}
|\Phi_{\hat \imath }(x^\star)|>0.
\end{equation}

Taking into account that the number of possible different choices of
${\cal A}^k$ and ${\cal N}^k$ is
finite, we can find a subset $\hat K\subseteq K\subseteq \{1,2,3,\dots\}$  such that ${\cal A}^k=\hat {\cal A}$ and ${\cal N}^k=\hat {\cal N}$ for  all $k\in \hat K$. We can have two different cases: either $\hat \imath \in \hat {\cal A}$ or $\hat \imath \in \hat {\cal N}$ for $k$ sufficiently large.

Suppose first that $\hat \imath \in \hat {\cal A}$ for $k$ sufficiently large. Then, by Definition~\ref{def:activeset}, we have for all $k\in \hat K$:
\begin{equation}\nonumber
 \max\{0,x_{\hat \imath}^k\}\leq \eps\,(g_{\hat \imath}(x^k)+\tau)\;
\mbox{ and } \;
 \max\{0,-x_{\hat \imath}^k\}\leq \eps\,(\tau-g_{\hat \imath}(x^k)).
\end{equation}
For all $k\in \hat K$, let $\tilde y^{{\hat \imath}, k}$ be the point defined
as in  (\ref{def-y}). By construction we have that
\begin{eqnarray}\label{prop-y}
&& \tilde y^{{\hat \imath}, k}_{\hat \imath}=0.
\end{eqnarray}
Now we consider three different subcases:
\begin{enumerate}
 \item[i)] $x_{\hat \imath}^k>0$. In this case, (\ref{def-y}) and (\ref{prop-y}) imply
 \begin{equation}\label{ineqcomp}
   (\tilde y_{\hat \imath}^{\hat \imath,k} - x_{\hat \imath}^k)\le 0.
 \end{equation}
 Recalling (\ref{eps-prop2}) in Assumption \ref{ass1}, there exists $\rho \ge 0$,  such that
 $$\epsilon \le  \frac{1}{H_{\hat \imath \hat \imath}+\rho}.$$
Furthermore, since $\hat \imath \in \hat {\cal A}$, we can write
\begin{eqnarray*}
 x_{\hat \imath}^k&\leq& \eps\,(g_{\hat \imath}(x^k)+\tau)\\ \nonumber
 x_{\hat \imath}^k-\tilde y^{\hat \imath,k}_{\hat \imath}&\leq& \eps\,(g_{\hat \imath}(x^k)+\tau)\\ \nonumber
 x_{\hat \imath}^k-\tilde y^{\hat \imath,k}_{\hat \imath}&\leq& \frac{1}{H_{\hat \imath \hat \imath}+\rho}(g_{\hat \imath}(x^k)+\tau) \nonumber
\end{eqnarray*}
Then we have:
\begin{equation}\nonumber
 (H_{\hat \imath \hat \imath}+\rho) (x_{\hat \imath}^k-\tilde y^{\hat \imath,k}_{\hat \imath})\leq g_{\hat \imath}(x^k)+\tau,
\end{equation}
which can be rewritten as follows
\begin{equation}\nonumber
  g_{\hat \imath}(x^k)+ H_{\hat \imath \hat \imath} (\tilde y^{\hat \imath,k}_{\hat \imath}-x_{\hat \imath}^k) +\tau\geq \rho (x_{\hat \imath}^k-\tilde y^{\hat \imath,k}_{\hat \imath})\geq 0,\\
\end{equation}
that is
\begin{equation}\label{ge0}
g_{\hat \imath}(\tilde y^{\hat \imath, k})+\tau\geq 0.
\end{equation}
On the other hand, since
$$0\le \max\{0,-x_{\hat \imath}^k\}\leq \eps\,(\tau-g_{\hat \imath}(x^k)) $$
we have that
$g_{\hat \imath}(x^k) - \tau \le 0$
and, as $H_{\hat \imath \hat \imath}\ge 0$ and \eqref{ineqcomp} holds, we get
\begin{equation}\label{le0}
g_{\hat \imath}(\tilde y^{\hat \imath, k})-\tau = g_{\hat \imath}(x^k)+ H_{\hat \imath \hat \imath} (\tilde y^{\hat \imath,k}_{\hat \imath}-x_{\hat \imath}^k) -\tau \le 0.
\end{equation}
By \eqref{prop-y}, \eqref{ge0} and \eqref{le0}, we have that
$$|\Phi_{\hat \imath}(\tilde y^{\hat \imath, k})|=0.$$
Furthermore, by \eqref{convytilde} and the continuity of $\Phi$, we can write
$$|\Phi_{\hat \imath}(x^\star)|=0.$$
Thus we get a contradiction with \eqref{cont0}.
\item[ii)] $x_{\hat \imath}^k<0$. It is a verbatim repetition of the previous case.
\item[iii)] $x_{\hat \imath}^k=0$. Since $\hat \imath \in \hat {\cal A}$ we have
$$g_{\hat \imath}(x^k)+\tau\geq 0 \;\; \mbox{ and } \; -(g_{\hat \imath}(x^k)-\tau)\geq 0,$$
which imply that
$$|\Phi_{\hat \imath}(x^k)|=0.$$
By the continuity of $\Phi(\cdot)$ and the fact that
\begin{equation}\nonumber\label{lim-x}\lim_{k \to \infty, k \in \hat K}x^k=x^\star,\end{equation}
 we get a contradiction with \eqref{cont0}.
\end{enumerate}
\par\medskip
Suppose now that $\hat \imath \in \hat {\cal N}$ for $k$ sufficiently large. We can choose a further subsequence $\{x^k\}_{\tilde K}$ with $\tilde K \subseteq \hat K$ such that
 $$|\Phi_{\bar \imath}(x^k)| = \max_{i\in \hat {\cal N}} |\Phi_i(x^k)|, \quad \forall \ k \in \tilde K.$$
Hence,
\begin{equation}\label{cont1}
|\Phi_{\bar \imath}(x^k)|\geq |\Phi_{\hat \imath}(x^k)|, \quad \forall \ k \in \tilde K,
\end{equation}
which, by  continuity of $\Phi(\cdot)$, implies
\begin{equation}\label{assur-2}
|\Phi_{\bar \imath}(x^\star)|\geq |\Phi_{\hat \imath}(x^\star)|.
\end{equation}
Furthermore, the instructions of Algorithm \texttt{FAST-BCDA} guarantee that,  for all $k \in \tilde K$, a set of indices $I_{h_k}$ exists such that
$$\bar \imath \in I_{h_k} \subseteq \bar {\cal N}^k_{ord}.$$
For all $k \in \tilde K$, Algorithm \texttt{FAST-BCDA}  produces a vector $y^{h_k, k}$ by minimizing Pro\-blem~(\ref{l2l1}) with respect to
all the variables whose indices belong to $I_{h_k}$. Therefore, the point $y^{h_k, k}$  satisfies
\begin{equation}\nonumber
 |\Phi_{\bar \imath}(y^{h_k, k})|=0.
\end{equation}
Furthermore, by \eqref{convy}, the continuity of $\Phi(\cdot)$, and taking into account \eqref{assur-2}, we can write
$$
0=|\Phi_{\bar \imath}(x^\star)|\geq |\Phi_{\hat \imath}(x^\star)|,
$$
which contradicts (\ref{cont0}).
\hfill $\Box$
\par\medskip\noindent

\section{Theoretical results related to the convergence rate analysis}\label{appconrate}
Here, following the ideas in \cite{LuoTseng}, we prove that the convergence rate of \texttt{FAST-BCDA} 
with 1-dimensional blocks (namely \texttt{FAST-1CDA}) is linear. 
First, we try to better analyze the indices in the set  ${\cal  N}(x)$ by introducing the following two sets:
\begin{equation}\label{eq:N^+(x)}
{\cal N}^+(x)=\{i \in {\cal N}(x): \ g_i(x)\leq 0  \},\ \mbox{and} \ {\cal N}^-(x)=\{i \in {\cal N}(x): \ g_i(x)> 0  \}. 
\end{equation}
We further introduce the sets:
\begin{equation}\label{eq:N^+(x^*)}
{ \cal E}^+(x^\star)=\{i: x^\star_i\geq0, \ g_i(x^\star)=-\tau  \},\ \mbox{and} \
{\cal  E}^-(x^\star)=\{i:\ x^\star_i\leq 0, \ g_i(x^\star)=\tau \},
\end{equation}
which satisfy the following equality:
$${ \cal E}(x^\star)={ \cal E}^+(x^\star)\cup { \cal E}^-(x^\star)= {\cal \bar N}(x^\star)\cup \{i: x^\star_i=0, \ |g_i(x^\star)|=\tau\}.$$
We further notice that 
\begin{equation}\label{indicesuse}
 I={\cal \bar A}^+(x^\star)\cup { \cal E}^+(x^\star) \cup { \cal E}^-(x^\star).
\end{equation}

We can finally prove a result that will be used in the convergence analysis:
\begin{theorem}\label{signs}
  Let $x^\star\in \R^n$ be a solution of Problem~\eqref{l2l1}. Then, there
exists a neighborhood of $x^\star$ such that, for each $x$ in this neighborhood, we have
\begin{eqnarray}
{\cal N}^+(x)&\subseteq& { \cal E}^+(x^\star), \label{s1}\\
{\cal N}^-(x)&\subseteq& { \cal E}^-(x^\star). \label{s2}
\end{eqnarray}
\end{theorem}
\begin{proof}
Let us assume there exists a sequence $\{\eps^k \}$, $\eps^k \to 0$, 
a related sequence of neighborhoods $\{{\cal B}(x^\star,\eps^k)\}$ and a sequence of 
points $\{x^k\}$ such that $x^k\in {\cal B}(x^\star,\eps^k)$ for all $k$, satisfying the following:
$${\cal N}^+(x^k)\not \subseteq { \cal E}^+(x^\star).$$
Then, since the number of indices is finite, there exist subseqences $\{\eps^k \}_K$ and $\{{\cal B}(x^\star,\eps^k)\}_K$ such that
an index $\hat \imath$ can be found, satisfying the following:
$$\hat \imath \in {\cal N}^+(x^k),\quad  \hat \imath \notin {\cal E}^+(x^\star).$$
From Theorem \ref{activeest}, for $k$ sufficiently large,
$${\cal N}(x^k)\subseteq { \cal E}^+(x^\star)\cup { \cal E}^-(x^\star).$$
Therefore, we have that
$$\hat \imath \in {\cal E}^-(x^\star) \;\mbox{ and } \; g_{\hat \imath}(x^\star) = \tau.$$
By continuity of the gradient, $g_{\hat \imath}(x^k)>0$ for $k$ sufficiently large.
On the other hand, since ${\hat \imath} \in {\cal N}^+(x^k)$, we have $g_{\hat \imath}(x^k)\leq 0$. 
This gives a contradiction, and proves \eqref{s1}.
A similar reasoning can be used for proving \eqref{s2}.
\hfill
\end{proof}
\par\medskip

Finally, we report another theoretical result that is used in the convergence rate analysis.
 
 \begin{proposition}\label{resultuseful}
  Let Assumption \ref{ass3} hold. Then, there exists a $\bar k$ such that\\
  \begin{itemize}
   \item[ a)] $x_i^k=0, \quad i \in {\cal \bar A}^+(x^\star)$;\\
   \item[ b)] $-sign(g_i(x^\star))~x_i^k\geq 0, \quad i \in {\cal E}(x^\star)$;\\
  \end{itemize}
  for all $k \geq \bar k$.
 \end{proposition}
 \par\medskip\noindent
 \begin{proof} a).  Recalling \eqref{stcmp}, for $k$ sufficently large, we have
  $${\cal \bar A}^+(x^\star) \subseteq {\cal A}^k.$$
  Therefore,  by taking into account the steps of \texttt{FAST-1CDA} Algorithm, we have
  \begin{equation}\label{acttozero}
   x_i^{k+1}=0, \quad i \in {\cal \bar A}^+(x^\star).  
   \end{equation}
  Furthermore, by continuity of $g$ and \eqref{convtoopt}, we obtain
  \begin{equation}\label{condstrc}
   \begin{array}{l}
  \tau+g_i(x^{k+1})>0, \quad \mbox{and} \quad \tau-g_i(x^{k+1})>0,
  \end{array}
  \end{equation}
and we can write $$i \in {\cal A}^{k+1}.$$
Hence, \eqref{acttozero} and \eqref{condstrc}  still hold for $x^{k+2}$, and so on.

b). Let us consider an index $i \in {\cal E}(x^\star)$. By contradiction, we assume that there exists 
a subsequence $K=\{k_1, k_2, \dots\}$ 
such that 
\begin{equation}\label{contrsign}
-sign(g_i(x^\star))~x_i^k< 0,
\end{equation}
for all $k \in K$. Without any loss of generality, we can consider
another subsequence $\hat K=\{\bar k_1, \bar k_2, \dots\}$ related to $K$, such that $i \in {\cal N}^{\bar k_j}$ and
\begin{equation}\label{transf}
 x_i^{k_j}= - sign\left(g_i(x^{\bar k_j})- H_{ii} x_i^{\bar k_j} \right)\ \frac{\max \left \{\left |g_i(x^{\bar  k_j})- H_{ii} x_i^{\bar k_j}\right|-\tau,0\right \}}{H_{ii}},
\end{equation}
for all $ k_j \in  K$ and $\bar k_j \in \hat K.$ 

If $i \in {\cal E}(x^\star)\setminus {\cal \bar N}(x^\star)$, when $j$ is 
sufficiently large, we  
have by  continuity of $g$, \eqref{convtoopt} and \eqref{transf}
$$
-sign(g_i(x^\star))~x_i^{k_j}\geq 0,
$$
which contradicts \eqref{contrsign}. 

If $i \in {\cal \bar N}(x^\star)$,  when $j$ is 
sufficiently large, we  
have by  continuity of $g$, \eqref{eq:N(x^*)} and \eqref{convtoopt} 
$$
-sign(g_i(x^\star))~x_i^{k_j}\geq 0,
$$
and again we get a contradiction with \eqref{contrsign}. 
 \hfill
 \end{proof}
\par\medskip

Now, we prove that the algorithm converges at linear rate.
\par\smallskip
\emph{Proof of  Theorem \ref{mainconvres}.}First of all, for ease of notation we set ${\cal \bar A}^+={\cal \bar A}^+(x^\star)$,  and ${\cal E}={\cal  E}(x^\star)$. Without any loss of generality, we can assume $|\bar{\cal N}_{ord}^k|=1$ for all $k$.
We then notice that  the objective function $f(x)$ can be rewritten as follows:
$$f(x)=q(x)+\tau \sum_{i \in {\cal \bar A}^+} sign(x_i)x_i+\tau \sum_{i \in {\cal E}} sign(x_i)x_i.$$
We further introduce the function 
$$F(x)=q(x)-\sum_{i \in {\cal E} } sign(g_i(x^\star)) x_i.$$
By taking into account Proposition \ref{resultuseful}, we have, for $k$ sufficiently large, 
\begin{equation}\label{propeq}
\begin{array}{l}
f(x^k)=F(x^k), \quad \mbox{and} \quad f(x^\star)=F(x^\star).
\end{array} 
\end{equation}
Furthermore, when $k$ is sufficiently large, by definition of ${\cal \bar A}^+$ and ${\cal E}$, and recalling again Proposition  
\ref{resultuseful}, we can write 
\begin{equation}\label{eq1pr}
x^k_{\bar{\cal A}^+}=x^\star_{\bar{\cal A}^+},
\end{equation}
\begin{equation}\label{eq2pr}
\nabla_i F(x^\star)=0, \quad \forall~i \in {\cal E}. 
\end{equation}
Then, by considering \eqref{indicesuse}, \eqref{eq1pr} and \eqref{eq2pr}, it follows
\begin{eqnarray*}
F(x^k)-F(x^\star)&=&\nabla F(x^\star)^\top(x^k-x^\star)+\frac{1}{2}(x^k-x^\star)^\top\nabla^2 F(x^\star)(x^k-x^\star) \\
                 &=&\frac{1}{2}(x^k-x^\star)^\top\nabla^2 F(x^\star)(x^k-x^\star)\leq \frac{\lambda_{max}(\nabla^2 F(x^\star))}{2} \|x^k-x^\star\|^2,\\              
\end{eqnarray*}
and, taking into account \eqref{propeq}, we can write
\begin{equation}\label{ineqkstar}
 f(x^k)-f(x^\star)\leq \rho \|x^k-x^\star\|^2,
\end{equation}
with $\rho>0$. Then, recalling Theorem \ref{activeest} and \ref{signs}, for $k$ sufficiently large the problem we actually solve is
\begin{equation}\label{problemmod2}
\begin{array}{ll}
\min &\tilde F(x)=\displaystyle\frac{1}{2}\|Ax-b\|^2-\tau~sign(g_{{\cal N}^k}(x^k))^\top x_{{\cal N}^k}\\
&x_{{\cal A}^k}=0\\
&-sign(g_i(x^k))~x_i\geq 0 \quad i \in {{\cal N}^k}.
\end{array}
\end{equation}
Now, let $y^{0,k}$ be the point obtained at Step 4 of Algorithm \ref{fig:FAST-CDA} (i.e. after fixing to zero the active variables) and $y^{0,k}_s$ the component that most violates 
condition \eqref{tsengest} in the non-active set.   We notice that finding the most violating variable according to condition \eqref{tsengest} is equivalent, when considering Problem \eqref{problemmod2}, to get the component 
that most violates the following condition
$$|x_i-[x_i-\nabla_i \tilde F(x)]_+|,$$
see \cite{LuoTseng} for further details.
Thus, we can write
\begin{eqnarray}\label{inequalitya}
 \frac{1}{\sqrt{|{{\cal N}^k}|}}\|y^{0,k}-[y^{0,k}-\nabla \tilde F(y^{0,k})]_+\|&\leq& |y^{0,k}_s-[y^{0,k}_s-\nabla_s \tilde F(y^{0,k})]_+| \nonumber\\ 
 & = & |y^{0,k}_s-[y^{0,k}_s-\nabla_s \tilde F(y^{0,k})]_+ -x^{k+1}_s+[x^{k+1}_s-\nabla_s \tilde F(x^{k+1})]_+| \nonumber \\
  &\leq& 2|y^{0,k}_s-x^{k+1}_s|+|\nabla_s \tilde F(y^{0,k})-\nabla_s \tilde F(x^{k+1})| \nonumber\\
  &\leq& 2\|y^{0,k}-x^{k+1}\|+\|\nabla \tilde F(y^{0,k})-\nabla \tilde F(x^{k+1})\|\nonumber \\
  &\leq& M\|y^{0,k}-x^{k+1}\|=M \|x^k-x^{k+1}\|_{{\cal N}^k}, 
\end{eqnarray}
where  $[\cdot]_+$ is the projection on the set of inequalities in Problem \eqref{problemmod2}, and \mbox{$M=\max \{2, L\}$}, with $L$ Lipschitz 
constant of $\nabla \tilde F$. By using Propositions \ref{prop1} and \ref{lemma2}, we can also write:
 \begin{eqnarray}\label{inequalityb}
  f(x^k)-f(x^{k+1})&\geq& \delta \|x^{k+1}-x^k\|^2 
 \end{eqnarray}
with $\delta>0$.
By taking into account inequality \eqref{inequalitya} and the definition of $y^{0,k}$, we can write, for $k$ sufficiently large, 
\begin{eqnarray}\label{distance2}
  \|x^{k+1}-x^k\|^2 &=& \|x^{k}-x^\star\|^2_{{\cal A}^k}+\|y^{0,k}-x^{k+1}\|^2_{{\cal N}^k}\nonumber\\
                  &\geq& \|x^{k}-x^\star\|^2_{{\cal A}^k}+ \frac{1}{M\sqrt{|{{\cal N}^k}|}}\|y^{0,k}-[y^{0,k}-\nabla \tilde F(y^{0,k})]_+\|^2.
                  \end{eqnarray}
Now, considering Theorem 2.1 in \cite{LuoTseng} we have, for $k$ sufficiently large, 
$$\sigma\|y^{0,k}-[y^{0,k}-\nabla \tilde F(y^{0,k})]_+\|\geq \|y^{0,k}-x^\star\|=\|x^k-x^\star\|_{{\cal N}^k},$$
with $\sigma>0$. Therefore, by taking into account inequality \eqref{distance2}, we can write
\begin{equation}\label{inequalityc}
 \|x^{k+1}-x^k\|^2 \geq \|x^{k}-x^{\star}\|^2_{{\cal A}^k}+\gamma\|x^{k}-x^{\star}\|^2_{{\cal N}^k}\geq \tilde\gamma\|x^{k}-x^{\star}\|^2,
\end{equation}
with $\tilde \gamma >0$. By combining inequalities \eqref{ineqkstar}, \eqref{inequalityb} and \eqref{inequalityc}, we can write
\begin{equation}\label{inequalityf}
f(x^k)-f(x^\star)\leq c_1 \left ( f(x^k)- f(x^{k+1}) \right ),
\end{equation}
with $c_1>1$. After rearranging the terms in \eqref{inequalityf}, we obtain 
$$f(x^{k+1})-f(x^\star)\leq  c_2 \left ( f(x^k)-f(x^\star) \right )$$
with $c_2= \left(1-\frac{1}{c_1}\right )<1.$ Then, $\{f(x^k)\}$ converges at least linearly to $f^\star$. 

Finally, by using \eqref{inequalityb} and Lemma 3.1 in \cite{LuoTseng} we get that the sequence $\{x^k\}$ 
converges at least linearly to $x^\star$.
\hfill$\Box$
\par

\end{document}